\documentclass{article}
\usepackage{amssymb, latexsym, graphicx, color, epsfig, amsmath, amsthm}

\input xy

\xyoption{all}

\newtheorem{theorem}{Theorem}[section]
\newtheorem{maintheorem}{Main Theorem}
\newtheorem{lemma}[theorem]{Lemma}
\newtheorem{proposition}[theorem]{Proposition}
\newtheorem{corollary}[theorem]{Corollary}
\newtheorem{definition}[theorem]{Definition}

\newtheorem{example}{\it Example\/}

\def\example{{\bf {\bigskip}{\noindent}Example: }}

\def\exampleone{{\bf {\bigskip}{\noindent}Example 1: }}

\newcommand{\s}{\ensuremath{\sigma}}

\newcommand{\F}{\ensuremath{\mathbb{F}}}

\newcommand{\cG}{\ensuremath{\mathcal{G}}}

\newcommand{\cC}{\ensuremath{\mathcal{C}}}
\newcommand{\cS}{\ensuremath{\mathcal{S}}}

\newcommand{\CAT}{\ensuremath{\operatorname{CAT}}}

\newcommand{\Aut}{\ensuremath{\operatorname{Aut}}}

\newcommand{\Vol}{\ensuremath{\operatorname{Vol}}}

\newcommand{\Ad}{\ensuremath{\operatorname{Ad}}}

\newcommand{\Stab}{\ensuremath{\operatorname{Stab}}}
\newcommand{\half}{\ensuremath{\operatorname{half}}}

\newcommand{\quot}{\backslash \! \backslash}

\newcommand{\G}{\Gamma}
\newcommand{\bs}{\backslash}

\def\polhk#1{\setbox0=\hbox{#1}{\ooalign{\hidewidth
    \lower1.0ex\hbox{$\,\lhook$}\hidewidth\crcr\unhbox0}}}
\newcommand{\Swiatkowski}{\'Swi{\polhk{a}}tkowski}

\begin{document}

\title{Existence, covolumes and infinite generation of lattices for Davis complexes}

\author{Anne Thomas\thanks{This work was supported in part by NSF Grant No. DMS-0805206 and in part by EPSRC Grant No. EP/D073626/2.  The author is currently supported by ARC Grant No. DP110100440.} \\ School of Mathematics and Statistics F07 \\ University of Sydney \\ Sydney NSW 2006 \\ Australia \\ anne.thomas@sydney.edu.au}

\date{revised 6 April 2010}

\maketitle

\begin{abstract} Let $\Sigma$ be the Davis complex for a Coxeter system $(W,S)$.  The automorphism group
$G$ of $\Sigma$ is naturally a locally compact group, and a simple combinatorial condition due to
Haglund--Paulin determines when $G$ is nondiscrete.  The Coxeter group $W$ may be regarded as a uniform
lattice in $G$.  We show that many such $G$ also admit a nonuniform lattice
$\G$, and an infinite family of uniform lattices with covolumes converging to that of $\G$.  It follows
that the set of covolumes of lattices in $G$ is nondiscrete.  We also show that the nonuniform lattice $\G$
is not finitely generated.  Examples of $\Sigma$ to which our results apply include buildings and non-buildings, and many complexes of dimension greater than $2$.  To prove these results, we introduce
a new tool, that of ``group actions on complexes of groups", and use this to construct our lattices as fundamental groups of complexes of groups with universal cover $\Sigma$. \end{abstract}

\section{Introduction}\label{s:intro}

Let $G$ be a locally compact topological group, with Haar measure $\mu$.  A discrete subgroup $\G \leq
G$ is a \emph{lattice} if $\G \bs G$ carries a finite $G$--invariant measure, and is \emph{uniform} if
$\G \bs G$ is compact.  Some basic questions are: \begin{enumerate} \item\label{q:existence} Does $G$
admit a (uniform or nonuniform) lattice? \item\label{q:covolumes} What is the set of \emph{covolumes}
of lattices in $G$, that is, the set of positive reals \[\mathcal{V}(G):=\{ \mu(\G\bs G) \mid \mbox{$\G
< G$ is a lattice}\}?\] \item\label{q:generation} Are lattices in $G$ finitely generated?
\end{enumerate}

These questions have been well-studied in classical cases.  For example, suppose $G$ is a reductive
algebraic group over a local field $K$ of characteristic $0$.  Then $G$ admits a uniform lattice,
constructed by arithmetic means (Borel--Harder~\cite{BoHa}), and a nonuniform lattice only if $K$ is
archimedean (Tamagawa~\cite{Ta}).  If $G$ is a semisimple real Lie group, the set $\mathcal{V}(G)$ is in most
cases discrete (see~\cite{Lu} and its references).  If in addition $G$ is simple and higher-rank, then $G$ and hence
its lattices have Kazhdan's Property (T) (see, for example,~\cite{Ma}).  Since countable groups with
Property (T) are finitely generated, it follows that all lattices in $G$ are finitely generated.

A nonclassical case is $G$ the automorphism group of a locally finite tree $T$.  The study of lattices in $G=\Aut(T)$ was initiated by Bass and Lubotzky, and has yielded many surprising differences from classical results (see the survey~\cite{Lu} and the reference~\cite{BL}).   For
example, the set $\mathcal{V}(G)$ is in many cases nondiscrete, and nonuniform
tree lattices are never finitely generated.

In fact, the automorphism group $G$ of any locally finite polyhedral complex $X$ is naturally a
locally compact group (see Section~\ref{ss:lattices}).  For many $X$ with $\dim(X)\geq 2$, there
is greater rigidity than for trees, as might be expected in higher dimensions.  For instance,
Burger--Mozes~\cite{BM} proved a `Normal Subgroup Theorem' for products of trees (parallel to that
of Margulis~\cite{Ma} for higher-rank semisimple Lie groups), and Bourdon--Pajot~\cite{BP} and
Xie~\cite{X} established quasi-isometric rigidity for certain Fuchsian buildings.  On the other
hand, lattices in $G=\Aut(X)$ can exhibit the same flexibility as tree lattices.  For example, the
set $\mathcal{V}(G)$ is nondiscrete for certain right-angled buildings~\cite{Th1} and Fuchsian
buildings~\cite{Th2}.  Another example is density of commensurators of uniform lattices in $G$,
proved by Haglund~\cite{H2} for certain $2$--dimensional Davis complexes, and by Haglund~\cite{H5}
and Kubena Barnhill--Thomas~\cite{KBT} for right-angled buildings.  Apart from right-angled
buildings, very little is known for $X$ of dimension $ > 2$.  Almost nothing is known for $X$ not a
building.

In this paper we consider Questions~\eqref{q:existence}--\eqref{q:generation} above for lattices
in $G=\Aut(\Sigma)$, where $\Sigma$ is the Davis complex for a Coxeter system $(W,S)$
(see~\cite{D} and Section~\ref{ss:Davis_complexes} below).  The Davis complex is a locally finite,
piecewise Euclidean $\CAT(0)$ polyhedral complex, and the Coxeter group $W$ may be regarded as a
uniform lattice in $G$.  Our results are the Main Theorem and its
Corollaries~\ref{c:nondiscreteness} and~\ref{c:infinite_generation} below, which establish
tree-like properties for lattices in many such $G$.  After stating these results, we discuss how they apply to (barycentric
subdivisions of) Fuchsian buildings and Platonic polygonal complexes, and to many Davis complexes $\Sigma$ with
$\dim(\Sigma) >
2$.

To state the Main Theorem, recall that for a Coxeter system $(W,S)$ with
$ W = \langle\, S \mid (st)^{m_{st}} \rangle$, and any $T \subset S$, the \emph{special subgroup} $W_T$ is the subgroup of $W$ generated by the elements $s \in T$.  A special subgroup $W_T$ is \emph{spherical} if
it is finite, and the set of spherical special subgroups of $W$ is partially ordered by inclusion.
The poset of nontrivial spherical special subgroups is an abstract simplicial complex $L$, called the
\emph{nerve} of $(W,S)$.  We identify each generator $s \in S$ with the corresponding vertex $W_{
\{s\} } = \langle s \rangle$ of $L$, and denote by $A$ the group of \emph{label-preserving
automorphisms} of $L$, that is, the group of automorphisms $\alpha$ of $L$ such that $m_{st} =
m_{\alpha(s)\alpha(t)}$ for all $s, t \in S$.  The group $G=\Aut(\Sigma)$ is nondiscrete if and only if there is a nontrivial $\alpha \in A$ such that $\alpha$ fixes the star in $L$ of some vertex $s$
(Haglund--Paulin~\cite{HP1}).

\begin{maintheorem}\label{t:existence} Let $(W,S)$ be a Coxeter system, with nerve $L$ and Davis
complex $\Sigma$.  Let $A$ be the group of label-preserving automorphisms of $L$.  Assume that
there are vertices $s_1$ and $s_2$ of $L$, and nontrivial elements $\alpha_1, \alpha_2 \in A$,
such that for $i = 1,2$: \begin{enumerate}  \item\label{c:fix} $\alpha_i$ fixes the star of
$s_{3-i}$ in $L$; \item\label{c:free} the subgroup $\langle \alpha_i \rangle$ of $A$ acts freely on the $\langle \alpha_{i} \rangle$--orbit of $s_i$, in particular $\alpha_i(s_{i}) \neq s_{i}$;  \item\label{c:orbit} for all $t_i \neq s_i$
such that $t_i$ is in the $\langle \alpha_{i} \rangle$--orbit of $s_i$, $m_{s_it_i} = \infty$; and
\item\label{c:halvable} all spherical special subgroups $W_T$ with $s_i \in T$  are \emph{halvable
along $s_i$} (see Definition~\ref{d:halvable} below).\end{enumerate} Then $G=\Aut(\Sigma)$ admits:
\begin{itemize} \item a nonuniform lattice $\G$; and \item an infinite family of uniform lattices
$(\G_n)$, such that $\mu(\G_n \bs G) \to \mu(\G \bs G)$, where $\mu$ is Haar measure on $G$.
\end{itemize}\end{maintheorem}

\begin{corollary}\label{c:nondiscreteness} The set of covolumes of
lattices in $G$ is nondiscrete. \end{corollary}

\begin{corollary}\label{c:infinite_generation} The group $G$ admits a lattice
which is not finitely generated. \end{corollary}

\noindent Corollary~\ref{c:infinite_generation} follows from the proof of the Main Theorem and
Theorem~\ref{t:group_action_intro} below.  By the discussion above,
Corollary~\ref{c:infinite_generation} implies that the group $G$ in the Main Theorem does not have
Property (T).  This was already known for $G=\Aut(\Sigma)$, where $\Sigma$ is any Davis complex (Haglund--Paulin~\cite{HP1}); our results thus provide an
alternative proof of this fact in some cases.

We describe several infinite families of examples of Davis complexes $\Sigma$ to
which our results apply in Section~\ref{s:examples} below.   To establish these
applications, we use properties of spherical buildings  in~\cite{R}, and some
results of graph theory from~\cite{DM}.  In two dimensions, examples include the Fuchsian
buildings considered in~\cite{Th2}, and some of the highly symmetric Platonic
polygonal complexes investigated by \Swiatkowski~\cite{Sw}.  Platonic polygonal
complexes are not in general buildings, and even the existence of lattices (other
than the Coxeter group $W$) in their automorphism groups was not previously known.  An example of a
Platonic polygonal complex is the (unique) $\CAT(0)$ $2$--complex with all
$2$--cells squares, and the link of every vertex the Petersen graph
(Figure~\ref{f:petersen} below).  The Main Theorem and its corollaries also apply to
many higher-dimensional $\Sigma$, including both buildings and complexes which are not buildings.

\begin{figure}[ht]
\begin{center}
\includegraphics{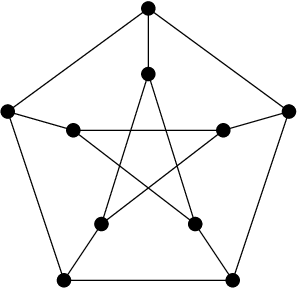}
\caption{Petersen graph}\label{f:petersen}
\end{center}
\end{figure}

To prove the Main Theorem, we construct the lattices $\G_n$ and $\G$ as fundamental groups of complexes of
groups with universal covers $\Sigma$ (see~\cite{BH} and Section~\ref{ss:complexes_of_groups} below).  The construction is given in Section~\ref{s:proof} below, where we also prove Corollary~\ref{c:infinite_generation}.

Complexes of groups are a generalisation to higher dimensions of graphs of groups. Briefly, given a polyhedral complex $Y$, a (simple) complex of groups $G(Y)$ over $Y$ is an assignment of a \emph{local group} $G_\s$ to each cell $\s$ of $Y$, with monomorphisms $G_\s \to G_\tau$ whenever $\tau \subset \s$, so that the obvious diagrams commute.  The action of a group $G$ on a polyhedral complex $X$ induces a complex of groups $G(Y)$ over $Y = G \bs X$.  A complex of groups is \emph{developable} if it is isomorphic to a complex of groups induced in this way.  A developable complex of groups $G(Y)$ has a simply-connected \emph{universal cover} $\widetilde{G(Y)}$, equipped with a canonical action of the \emph{fundamental group of the complex of groups} $\pi_1(G(Y))$.

A key difference from graphs of groups is that complexes of groups are not in general
developable.  In addition, even if $G(Y)$ is developable, with universal cover say $X$, it may be impossible to identify $X$ of dimension $\geq 2$ using only local data such as the links of
its vertices (see Ballmann--Brin~\cite{BB1} and Haglund~\cite{H1}).  To ensure that our complexes
of groups are developable with universal cover $\Sigma$, we use covering theory for complexes of
groups (see~\cite{BH} and~\cite{LT}, and Section~\ref{ss:definitions} below).  The
main result needed is that if there is a covering of complexes of groups $G(Y) \to H(Z)$, then
$G(Y)$ is developable if and only if $H(Z)$ is developable, and the universal covers of $G(Y)$ and
$H(Z)$ are isometric (see Theorem~\ref{t:coverings} below).

The other main ingredient in the proof of the Main Theorem is Theorem~\ref{t:group_action_intro} below, which introduces a theory of ``group actions on complexes of groups".  This is a method of manufacturing new complexes of groups with a given universal cover, by acting on previously-constructed complexes of groups.  Given a complex of groups $G(Y)$, and the action of a group $H$ on $Y$, the $H$--action \emph{extends to an action on $G(Y)$} if there is a homomorphism from $H$ to $\Aut(G(Y))$.  Roughly, this means that for each cell $\s$ of $Y$, each $h \in H$ induces a group isomorphism $G_\s \to G_{h\cdot\s}$, so that the obvious diagrams commute (see Section~\ref{ss:definitions} below for definitions).  In Section~\ref{s:group_actions} below we prove:

\begin{theorem}\label{t:group_action_intro}  Let $G(Y)$ be a (simple) complex of groups over $Y$, and suppose that the action of a group $H$ on $Y$ extends to an action on $G(Y)$.  Then the $H$--action induces a complex of groups $H(Z)$ over $Z = H \bs Y$ such that there is a covering of complexes of groups $G(Y) \to H(Z)$. Moreover there is a natural short exact sequence
\[ 1 \to \pi_1(G(Y)) \to \pi_1(H(Z)) \to H \to 1,\] and if $H$ fixes a vertex of $Y$, then
\[ \pi_1(H(Z)) \cong \pi_1(G(Y)) \rtimes H.\]
 \end{theorem}

\noindent Theorem~\ref{t:group_action_intro} is also used in~\cite{KBT}, and we expect this result
to be of independent interest.  To our knowledge, group actions on complexes of groups have not
previously been considered.  In~\cite{BJ}, Bass--Jiang determined the structure of the full
automorphism group of a graph of groups, but did not define or study the graph of groups induced
by a group action on a graph of groups.    A more precise statement of
Theorem~\ref{t:group_action_intro}, including some additional facts about $H(Z)$,  is given as
Theorem~\ref{t:group_action} below.

The Main Theorem is proved as follows.  The action of the Coxeter group $W$ on
$\Sigma$ induces a complex of groups $G(Y_1)$ over $Y_1= W \bs \Sigma$, with local
groups the spherical special subgroups of $W$.  We then construct a family of finite
complexes of groups $G(Y_n)$ and $H(Z_n)$, and two infinite complexes of groups
$G(Y_\infty)$ and $H(Z_\infty)$, so that there are coverings of complexes of groups
as sketched in Figure~\ref{f:coverings} below.

\begin{figure}[ht]\begin{center} \includegraphics{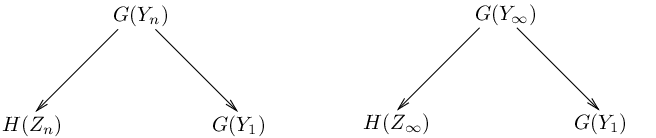} \caption{Coverings of
complexes of groups}\label{f:coverings} \end{center} \end{figure}

\noindent The fundamental groups of $H(Z_n)$ and $H(Z_\infty)$ are, respectively,
the uniform lattices $\G_n$, and the nonuniform lattice $\G$, in $G=\Aut(\Sigma)$.
For the local groups of $G(Y_n)$ and $G(Y_\infty)$, we use
Condition~\eqref{c:halvable} in the Main Theorem to replace certain spherical
special subgroups $W_T$ by the subgroup $\half_s(W_T)$, defined as follows:

\begin{definition}\label{d:halvable} Let $W_T$ be a spherical special subgroup of $W$, and suppose $s \in T$.  Then $W_T$ is \emph{halvable along $s$} if the union \[(T - \{ s \}) \cup \{ sts \mid t \in T - \{s\} \} \] generates an index $2$ subgroup, denoted $\half_s(W_T)$, of $W_T$.\end{definition}

\noindent The complexes of groups $H(Z_n)$ and $H(Z_\infty)$ are induced by group actions on,
respectively, $G(Y_n)$ and $G(Y_\infty)$.  To construct these group actions, we use
the automorphisms $\alpha_1$ and $\alpha_2$ of $L$.

I am grateful to Benson Farb for introducing me to these questions, and for his continuing
encouragement and advice.  I also thank G. Christopher Hruska and Kevin Wortman for many useful
discussions.  This particular work was inspired by conversations with Tadeusz Januszkiewicz and
Jacek \Swiatkowski, and much of this project was carried out at the Mathematical Sciences
Research Institute in Fall 2007, where I benefited from talking with Angela Kubena Barnhill, Michael W.
Davis, Jonathan P. McCammond and Damian Osajda.  I would also like to thank Karen Vogtmann for helpful comments on this manuscript, and an anonymous referee for careful reading and worthwhile suggestions.

\section{Background}\label{s:background}

In this section we present brief background.  In Section~\ref{ss:lattices} we describe the natural topology on $G$ the group of
automorphisms of a locally finite polyhedral complex $X$, and characterise uniform and nonuniform
lattices in $G$.   Section~\ref{ss:Davis_complexes} constructs the Davis complex $\Sigma$ for a
Coxeter system $(W,S)$, following~\cite{D}.  In Section~\ref{ss:complexes_of_groups} we recall the basics of Haefliger's
theory of complexes of groups (see~\cite{BH}), and describe the complex of groups $G(Y_1)$ induced by
the action of $W$ on $\Sigma$.

\subsection{Lattices in automorphism groups of polyhedral complexes}\label{ss:lattices}

Let $G$ be a locally compact topological group.  We will use the following normalisation of Haar
measure $\mu$ on $G$.

\begin{theorem}[Serre, \cite{S}]\label{t:Scovolumes} Suppose that a locally compact group $G$ acts on a set $V$ with compact open
stabilisers and a finite quotient $G\backslash V$.  Then there is a normalisation of the Haar measure
$\mu$ on $G$, depending only on the choice of $G$--set $V$, such that for each discrete subgroup
$\Gamma$ of $G$ we have \[\mu(\Gamma \bs G) = \Vol(\Gamma \quot V):=  \sum_{v \in \Gamma \bs V }
\frac{1}{|\Gamma_v|} \,\leq\infty. \] \end{theorem}

Suppose $X$ is a connected, locally finite polyhedral complex.  Let $G=\Aut(X)$.  With the compact-open
topology, $G$ is naturally a locally compact topological group, and the $G$--stabilisers of cells in $X$ are compact and open.
Hence if $G \bs X$ is finite, there are several natural choices of sets $V$ in
Theorem~\ref{t:Scovolumes} above.  By the same arguments as for tree lattices~(\cite{BL}, Chapter 1),
it can be shown (for any suitable set $V$) that a discrete subgroup $\G \leq G$ is a lattice if and only if the series $\Vol(\G
\quot V)$ converges, and $\G$ is uniform if and only if this is a sum with finitely many terms.  In
Section~\ref{ss:Davis_complexes} below we describe our choice of $G$--set $V$ when $G$ is the group
of automorphisms of a Davis complex $\Sigma$.

\subsection{Davis complexes}\label{ss:Davis_complexes}

In this section we recall the construction of the Davis complex for a Coxeter system.
We follow the reference~\cite{D}.

A \emph{Coxeter group} is a group $W$ with a finite generating set $S$ and presentation of the form
\[ W = \langle s\in S \mid s^2 = 1 \,\, \forall \, s \in S, (s t )^{m_{st}} = 1 \,\,\forall \, s,t \in S \mbox{ with
}s \neq t\rangle \] with $m_{st}$ an integer $\geq 2$
or $m_{st} = \infty$, meaning that $st$ has infinite order.   The pair $(W,S)$ is
called a \emph{Coxeter system}.

\exampleone   This example will be followed throughout this section, and also referred to in
Sections~\ref{ss:complexes_of_groups} and~\ref{s:proof} below.  Let
\[ W = \langle s_1,s_2,s_3,s_4,s_5 \mid s_i^2=1, (s_1s_4)^m=(s_2s_4)^m=(s_3s_4)^m = 1,\] \[
(s_1s_5)^{m'}=(s_2s_5)^{m'}=(s_3s_5)^{m'} = 1\rangle \]
where $m$ and $m'$ are integers $\geq 2$.

\vspace{3mm}

Let $(W,S)$ be a Coxeter system.  A subset $T$ of $S$ is \emph{spherical} if the corresponding special
subgroup $W_T$ is spherical, that is, $W_T$ is finite.  By convention, $W_\emptyset$ is the trivial
group.  Denote by $\cS$ the set of all spherical subsets of $S$.  The set $\cS$ is partially ordered by
inclusion, and the poset $\cS_{> \emptyset}$ is the nerve $L$ of $(W,S)$ (this is equivalent to the
description of $L$ in the introduction above).  By definition, a nonempty set $T$ of vertices of $L$ spans a
simplex $\s_T$ in $L$ if and only if $T$ is spherical.  We identify the generator $s \in S$ with the
vertex $\{s \}$ of $L$.  The vertices $s$ and $t$ of $L$ are joined by an edge in $L$ if and only
if $m_{st}$ is finite, in which case we label this edge by the integer $m_{st}$.  The nerve $L$ of
Example 1 above, with its edge labels, is sketched in Figure~\ref{f:nerve} below.

\begin{figure}[ht]
\begin{center}
\includegraphics{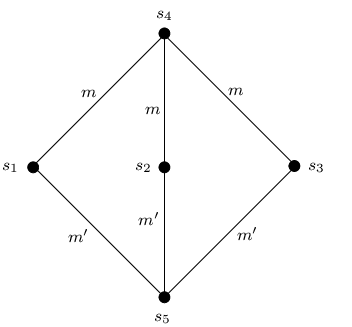}
\end{center}
\caption{Nerve $L$ of Example 1, with edge labels}
\label{f:nerve}
\end{figure}

We denote by $K$ the geometric realisation of the poset $\cS$.  Equivalently, $K$ is the cone on the
barycentric subdivision of the nerve $L$ of $(W,S)$.  Note that $K$ is compact and contractible, since
it is the cone on a finite simplicial complex.  Each vertex of $K$ has
\emph{type} a spherical subset of $S$, with the cone point having type $\emptyset$.  For Example 1 above,
$K$ and the types of its vertices are sketched on the left of Figure~\ref{f:K}.

\begin{figure}[ht]
\begin{center}
\resizebox{125mm}{!}{\epsfbox{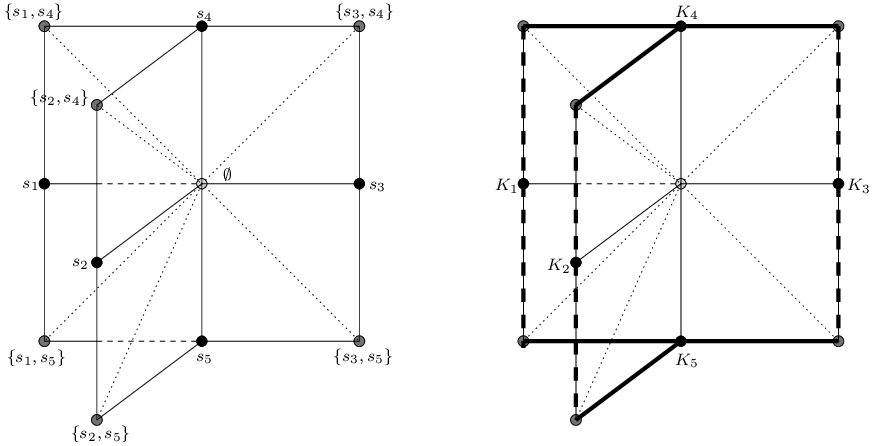}}
\end{center}
\caption{$K$ and types of vertices (left) and mirrors (right) for Example 1}
\label{f:K}
\end{figure}

For each $s \in S$ let $K_s$ be the union of the (closed) simplices in $K$ which contain the
vertex $s$ but not the cone point.  In other words, $K_s$ is the closed star of the vertex $s$
in the barycentric subdivision of $L$.   Note that $K_s$ and $K_t$ intersect if and only if
$m_{st}$ is finite.  The family $(K_s)_{s \in S}$ is a \emph{mirror structure} on $K$, meaning
that $(K_s)_{s\in S}$  is a family of closed subspaces of $K$, called \emph{mirrors}.  We call
$K_s$ the \emph{$s$--mirror} of $K$.  For Example 1 above, the mirrors $K_i=K_{s_i}$ are shown
in heavy lines on the right of Figure~\ref{f:K}.

For each $x \in K$, put \[S(x):= \{ s \in S \mid x \in K_s \}.\] Now define an equivalence relation
$\sim$ on the set $W \times K$ by $(w,x) \sim (w',x')$ if and only if $x = x'$ and $w^{-1}w' \in
W_{S(x)}$.  The \emph{Davis complex} $\Sigma$ for $(W,S)$ is then the quotient space: \[ \Sigma := (W
\times K) / \sim. \]  The types of vertices of $K$ induce types of the vertices of $\Sigma$, and the
natural $W$--action on $W \times K$ descends to a type-preserving action on $\Sigma$, with compact quotient $K$, so
that the $W$--stabiliser of a vertex of $\Sigma$ of type $T \in \cS$ is a conjugate of the spherical
special subgroup $W_T$.

We identify $K$
with the subcomplex $(1,K)$ of $\Sigma$, and write $wK$ for the translate $(w,K)$, where $w \in W$.
Any $wK$ is called a \emph{chamber} of $\Sigma$.  The mirrors $K_s$ of $K$, or any of their
translates by elements of $W$, are called the \emph{mirrors} of $\Sigma$.  Two distinct chambers of $\Sigma$
are \emph{$s$--adjacent} if their intersection is an $s$--mirror, and are \emph{adjacent} if their
intersection is an $s$--mirror for some $s \in S$.  Note that the chambers $wK$ and $w'K$ are
$s$--adjacent if and only if $w^{-1}w' = s$, equivalently $w' = ws$ and $w's = w$.  For Example 1 above, part of the Davis complex $\Sigma$ for $(W,S)$ is shown in
Figure~\ref{f:davis} below.  There are $2m$ copies of $K$ glued around the vertices of types
$\{s_i,s_4\}$, for $i = 1,2,3$, since $W_{\{s_i,s_4\}}$ has order $2m$.  Similarly, there are $2m'$ copies of $K$ glued around the vertices of types $\{s_i,s_5\}$, for $i = 1,2,3$.

The Davis complex $\Sigma$ may be metrised with a piecewise Euclidean structure, such that $\Sigma$
is a complete $\CAT(0)$ space (Moussong, see Theorem 12.3.3 of \cite{D}).  From now on we will assume
that $\Sigma$ is equipped with this metric.

\begin{figure}[ht]
\begin{center}
\includegraphics{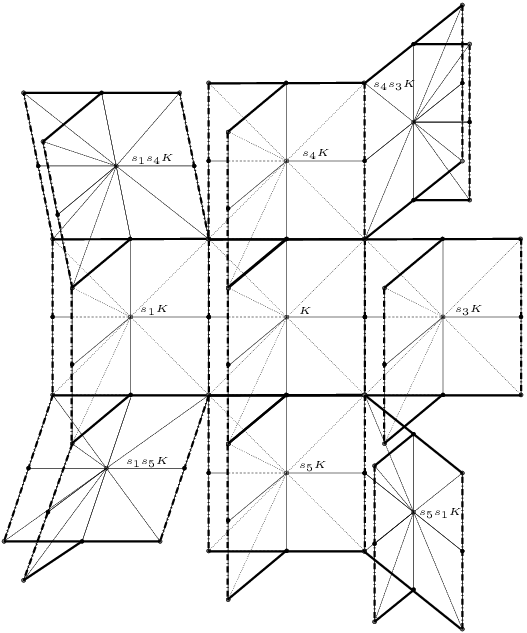}
\end{center}
\caption{Part of $\Sigma$, for Example 1}
\label{f:davis}
\end{figure}

Suppose that $G=\Aut(\Sigma)$ is the group of automorphisms of a Davis complex $\Sigma$.  Since $W$ acts cocompactly on $\Sigma$, with finite stabilisers, it may be regarded as a uniform
lattice in $G$.  We take as the set $V$ in Theorem~\ref{t:Scovolumes} above the set of vertices of $\Sigma$ of type $\emptyset$.  Then the covolume of $W$ is $1$, since $W$ acts simply transitively on $V$.

\subsection{Complexes of groups}\label{ss:complexes_of_groups}

We now outline the basic theory of complexes of groups, following Chapter III.$\cC$ of~\cite{BH}.  The definitions of the more involved notions of morphisms and coverings of complexes of groups are postponed to Section~\ref{ss:definitions} below. 

In the literature, a complex of groups $G(Y)$ is constructed over a space or set $Y$ belonging to
various different categories, including simplicial complexes, polyhedral complexes, or, most
generally, \emph{scwols} (small categories without loops).  In each case there is a set of
vertices, and a set of oriented edges with a composition law.  The formal definition of a scwol is:

\begin{definition}\label{d:scwol} A \emph{scwol} $X$ is the disjoint
union of a set $V(X)$ of vertices and a set $E(X)$ of edges, with each edge $a$ oriented from its
initial vertex $i(a)$ to its terminal vertex $t(a)$, such that $i(a) \not = t(a)$.  A pair of
edges $(a,b)$ is \emph{composable} if $i(a)=t(b)$, in which case there is a third edge $ab$, called
the \emph{composition} of $a$ and $b$, such that $i(ab)=i(b)$, $t(ab)=t(a)$ and if $i(a) = t(b)$
and $i(b)=t(c)$ then $(ab)c = a(bc)$ (associativity).  \end{definition}

\noindent We will always assume scwols are \emph{connected} (see Section~III.$\mathcal{C}$.1.1 of~\cite{BH}).  Morphisms of scwols and group actions on scwols are defined as follows:

\begin{definition}\label{d:morphism_scwols}  Let $X$, $Y$ and $Z$ be scwols.  A \emph{nondegenerate morphism} $f:Y \to Z$ is a map that sends $V(Y)$ to $V(Z)$, sends $E(Y)$ to $E(Z)$ and is such that:
\begin{enumerate}
\item for each $a \in E(Y)$, we have $i(f(a)) = f(i(a))$ and $t(f(a)) = f(t(a))$;
\item for each pair of composable edges $(a,b)$ in $Y$, we have $f(ab) = f(a)f(b)$; and
\item for each vertex $\s \in V(Y)$, the restriction of $f$ to the set of edges with initial vertex $\s$ is a bijection onto the set of edges of $Z$ with initial vertex $f(\s)$.
\end{enumerate}
A \emph{morphism of scwols} $f:Y \to Z$ is a functor from the category $Y$ to the category $Z$ (see Section~III.$\mathcal{C}$.A.1 of~\cite{BH}). An \emph{automorphism} of a scwol $X$ is a morphism from $X$ to $X$ that has an inverse.
\end{definition}

\begin{definition}\label{d:action_on_scwol} An \emph{action} of a
group $G$ on a scwol $X$ is a homomorphism from $G$ to the group of automorphisms of $X$ such that for all $a \in E(X)$ and all $g \in G$: \begin{enumerate}
\item $g\cdot i(a) \not = t(a)$; and \item \label{i:no_inversions} if $g \cdot i(a) = i(a)$ then $g \cdot a = a$.
\end{enumerate} \end{definition}

Suppose now that $\Sigma$ is the Davis complex for a Coxeter system $(W,S)$, as defined in Section~\ref{ss:Davis_complexes} above.  Recall that each vertex $\sigma \in V(\Sigma)$ has type $T$ a spherical subset of $S$.  The edges $E(\Sigma)$ are then naturally oriented by inclusion of type.  That is, if the edge $a$ joins the vertex $\sigma$ of type $T$ to
the vertex $\sigma'$ of type $T'$, then $i(a)=\sigma$ and $t(a)=\sigma'$ exactly when $T
\subsetneq T'$.  It is clear that the sets $V(\Sigma)$ and $E(\Sigma)$ satisfy the properties of a scwol.
Moreover, if $Y$ is a subcomplex of $\Sigma$, then the sets $V(Y)$ and $E(Y)$ also satisfy
Definition~\ref{d:scwol} above.  By abuse of notation, we identify $\Sigma$ and $Y$ with the associated
scwols.

We now define complexes of groups over scwols.

\begin{definition}
A \emph{complex of groups} $G(Y)=(G_\sigma, \psi_a, g_{a,b})$ over a scwol
$Y$ is given by: \begin{enumerate} \item a group $G_\sigma$ for each
$\sigma \in V(Y)$, called the \emph{local group} at $\sigma$;
\item a monomorphism $\psi_a: G_{i(a)}\rightarrow G_{t(a)}$ along the edge $a$ for each
$a \in E(Y)$; and
\item for each pair of composable edges, a twisting element $g_{a,b} \in
G_{t(a)}$, such that \[ \Ad(g_{a,b})\circ\psi_{ab} = \psi_a
\circ\psi_b
\] where $\Ad(g_{a,b})$ is conjugation by $g_{a,b}$ in $G_{t(a)}$,
and for each triple of composable edges $a,b,c$ the following
\emph{cocycle condition} holds
\[\psi_a(g_{b,c})\,g_{a,bc} = g_{a,b}\,g_{ab,c}.\] \end{enumerate}
\end{definition}

\noindent A complex of groups is \emph{simple} if each $g_{a,b}$ is trivial.

\example This example will be followed throughout this section, and used in the proof of the Main
Theorem in Section~\ref{s:proof} below.  Let $(W,S)$ be a Coxeter system with nerve $L$ and let $K$ be
the cone on the barycentric subdivision of $L$, as in Section~\ref{ss:Davis_complexes} above.  Put $Y_1
= K$, with the orientations on edges discussed above.  We construct a simple complex of groups $G(Y_1)$
over $Y_1$ as follows.  Let $\s \in V(Y_1)$.  Then $\s$ has type a spherical subset $T$ of $S$, and we
define $G_\s = W_T$.  All monomorphisms along edges of $Y_1$ are then natural inclusions, and all $g_{a,b}$
are trivial.  For $(W,S)$ as in Example 1 of Section~\ref{ss:Davis_complexes} above, the complex of
groups $G(Y_1)$ is shown in Figure~\ref{f:GWK} below.  In this figure, $D_{2m}$ and $D_{2m'}$ are the
dihedral groups of orders $2m$ and $2m'$ respectively, and $C_2$ is the cyclic group of order $2$.

\vspace{3mm}

\begin{figure}[ht]
\begin{center}
\includegraphics{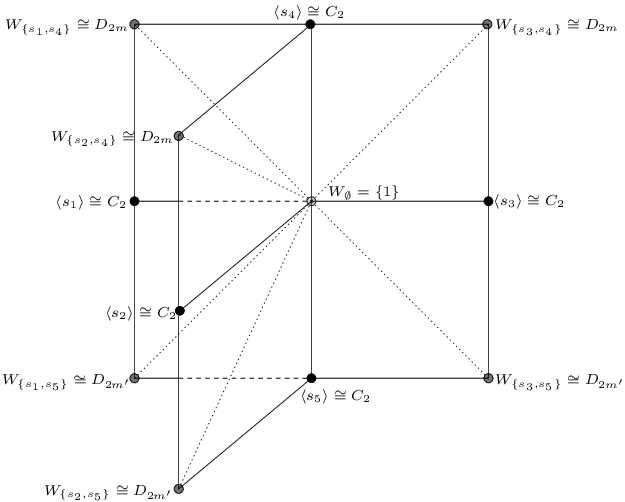}
\end{center}
\caption{The complex of groups $G(Y_1)$, for Example 1 of Section~\ref{ss:Davis_complexes}}
\label{f:GWK}
\end{figure}

Suppose a group $G$ acts on a scwol $X$, as in Definition~\ref{d:action_on_scwol} above.  Then the
quotient $Y = G \bs X$ also has the structure of a scwol, and the action of $G$ induces a complex of
groups $G(Y)$ over $Y$ (this construction is generalised in Section~\ref{ss:induced_cogs} below).  Let $Y =G \bs X$ with $p: X \to Y$  the natural projection. For each $\sigma \in V(Y)$, choose a lift $\overline\sigma \in V(X)$ such that $p(\overline\sigma) = \sigma$. The local group $G_\sigma$ of $G(Y)$ is then defined to be the stabiliser of $\overline\sigma$ in $G$, and the monomorphisms $\psi_a$ and elements $g_{a,b}$ are defined using further choices. The resulting complex of groups $G(Y)$ is unique up to
isomorphism (see Definition~\ref{d:morphism} below).

A complex of groups is \emph{developable} if it is isomorphic to a complex of groups $G(Y)$ induced, as
just described, by such an action.  Complexes of groups, unlike graphs of groups, are not in general
developable.  The complex of groups $G(Y_1)$ above is developable, since it is the complex of groups
induced by the action of $W$ on $\Sigma$, where for each $\s \in V(Y_1)$, with $\s$ of type $T$, we
choose $\overline\sigma$ in $\Sigma$ to be the vertex of $(1,K) = K \subset \Sigma$ of type $T$.

Let $G(Y)$ be a complex of groups.  The \emph{fundamental group of the complex of groups} $\pi_1(G(Y))$
is defined so that if $Y$ is simply connected and $G(Y)$ is simple, $\pi_1(G(Y))$ is isomorphic to the
direct limit of the family of groups $G_\sigma$ and monomorphisms $\psi_a$.  For example, since $Y_1 =
K$ is contractible and $G(Y_1)$ is a simple complex of groups, the fundamental group of
$G(Y_1)$ is $W$.

If $G(Y)$ is a developable complex of groups, then it has a \emph{universal cover}
$\widetilde{G(Y)}$.  This is a connected, simply-connected scwol, equipped with an action of
$\pi_1(G(Y))$, so that the complex of groups induced by the action of the fundamental group on the
universal cover is isomorphic to $G(Y)$.  For example, the universal cover of $G(Y_1)$ is $\Sigma$.

Let $G(Y)$ be a developable complex of groups over $Y$, with universal cover $X$ and fundamental
group $\G$.  We say that $G(Y)$ is \emph{faithful} if the action of $\G$ on $X$ is faithful, in which
case we may identify $\G$ with its image in $\Aut(X)$.  If $X$ is locally finite, then with the
compact-open topology on $\Aut(X)$, by the discussion in Section~\ref{ss:lattices} above the subgroup
$\Gamma \leq \Aut(X)$ is discrete if and only if all local groups of $G(Y)$ are finite.  Moreover, if
$\Aut(X)$ acts cocompactly on $X$, a discrete $\G \leq \Aut(X)$
is a uniform lattice in $\Aut(X)$ if and only if $Y \cong \G \bs X$ is finite, and a discrete $\G \leq \Aut(X)$ is a
nonuniform lattice if and only if $Y \cong \G \bs X$ is infinite and the series $\Vol(\G \quot V)$ converges, for
some $V \subset X$ on which $G=\Aut(X)$ acts according to the hypotheses of
Theorem~\ref{t:Scovolumes} above.

\section{Group actions on complexes of groups}\label{s:group_actions}

In this section we introduce a theory of group actions on complexes of groups.  The main result is
Theorem~\ref{t:group_action} below, which makes precise and expands upon
Theorem~\ref{t:group_action_intro} of the introduction.  The terms appearing in Theorem~\ref{t:group_action}
which have not already been discussed in Section~\ref{ss:complexes_of_groups} above will be defined in
Section~\ref{ss:definitions} below, where we also introduce some notation.  In
Section~\ref{ss:induced_cogs} below we construct the complex of groups induced by a group action on a complex
of groups, and in Section~\ref{ss:induced_covering} we construct the induced covering.  Using these
results, in Section~\ref{ss:fund_group} we consider the structure of the fundamental group of the
induced complex of groups. 

We will require only actions on \emph{simple} complexes of groups $G(Y)$ by \emph{simple} morphisms; this case is already substantially technical.  If in addition the action on $Y$ has a strict fundamental domain, it is possible to make choices so that the induced complex of groups is also simple, and many of the proofs in this section become much easier.  However, in our applications, the group action will not necessarily have a strict fundamental domain.

\begin{theorem}\label{t:group_action}  Let $G(Y)$ be a simple complex of groups over a connected scwol
$Y$, and suppose that the action of a group $H$ on $Y$ extends to an action by simple morphisms on
$G(Y)$.  Then the $H$--action induces a complex of groups $H(Z)$ over $Z = H \bs Y$, with $H(Z)$
well-defined up to isomorphism of complexes of groups, such that there is a covering of complexes of
groups \[ G(Y) \to H(Z). \] Moreover there is a natural short exact sequence \[1 \to \pi_1(G(Y)) \to \pi_1(H(Z)) \to H \to 1,\] and if $H$ fixes a vertex of $Y$, then
\[ \pi_1(H(Z)) \cong \pi_1(G(Y)) \rtimes H.\] Finally, if $G(Y)$ is faithful and the $H$--action on $G(Y)$ is faithful then
$H(Z)$ is faithful.\end{theorem}

We will use the following general result on functoriality of coverings (which is implicit in~\cite{BH}, and stated and proved explicitly in~\cite{LT}).

\begin{theorem}\label{t:coverings} Let $G(Y)$ and $H(Z)$ be complexes of groups over scwols $Y$ and
$Z$.  Suppose there is a covering of complexes of groups $\Phi:G(Y) \to H(Z)$.  Then $G(Y)$ is
developable if and only if $H(Z)$ is developable. Moreover, $\Phi$ induces a monomorphism of fundamental groups \[  \pi_1(G(Y)) \hookrightarrow \pi_1(H(Z))\]
and an equivariant isomorphism of universal covers \[ \widetilde{G(Y)} \stackrel{\cong}{\longrightarrow}
\widetilde{H(Z)}.\] \end{theorem}

\subsection{Definitions and notation}\label{ss:definitions}

We gather here definitions and notation needed for the statement and proof of
Theorem~\ref{t:group_action} above.  Throughout this section, $Y$ and $Z$ are scwols, $G(Y)=(G_\s,\psi_a)$ is a simple
complex of groups over $Y$, and $H(Z)=(H_\tau,\theta_a, h_{a,b})$ is a complex of groups over
$Z$.

\begin{definition}\label{d:morphism} Let $f: Y\to Z$ be a morphism of scwols (see Definition~\ref{d:morphism_scwols} above).  A \emph{morphism} $\Phi: G(Y) \to H(Z)$ over $f$ consists of: \begin{enumerate}
\item a homomorphism $\phi_\sigma: G_\sigma \to H_{f(\sigma)}$ for each $\sigma \in V(Y)$, called
the \emph{local map} at $\s$; and
\item\label{i:commuting} an element $\phi(a) \in H_{t(f(a))}$ for each $a \in E(Y)$, such that the following diagram commutes
\[\xymatrix{
G_{i(a)}   \ar[d]^-{\phi_{i(a)}} \ar[rrr]^{\psi_a} & & & G_{t(a)} \ar[d]^-{\phi_{t(a)}}
\\
H_{f(i(a))}  \ar[rrr]^{\Ad(\phi(a))\circ \theta_{f(a)}} & & & H_{f(t(a))}
}\]
and for all pairs of
composable edges $(a,b)$ in $E(Y)$, \[ \phi(ab) = \phi(a) \,\psi_a(\phi(b))h_{f(a),f(b)} . \]
\end{enumerate} \end{definition}

\noindent A morphism is \emph{simple} if each element $\phi(a)$ is trivial.  If $f$ is an isomorphism of scwols, and each $\phi_\sigma$ an isomorphism of the local groups, then $\Phi$ is an \emph{isomorphism of complexes of groups}.

We introduce the following, expected, definitions.  An \emph{automorphism} of $G(Y)$ is an
isomorphism $\Phi:G(Y) \to G(Y)$.  It is not hard to verify that the set of automorphisms of
$G(Y)$ forms a group under composition, which we denote $\Aut(G(Y))$ (see Section~III.$\mathcal{C}$.2.4 of~\cite{BH}
for the definition of composition of morphisms).  We then say that a group \emph{$H$ acts
on $G(Y)$} if there is a homomorphism \[ H \to \Aut(G(Y)).\]

Our notation is as follows.  Suppose $H$ acts on $G(Y)$.  Then in particular, $H$ acts on the
scwol $Y$ in the sense of Definition~\ref{d:action_on_scwol} above.  We write the action of $H$ on $Y$ as
$\s \mapsto h\cdot\s$ and $a \mapsto h\cdot a$, for $h \in H$, $\s \in V(Y)$ and $a \in E(Y)$.  The element
$h \in H$ induces the automorphism $\Phi^h$ of $G(Y)$.  The data for
$\Phi^h$ is a family $(\phi^h_\sigma)_{\sigma \in V(Y)}$ of group isomorphisms
$\phi^h_\sigma:G_\sigma \to G_{h\cdot\s}$, and a family of elements $(\phi^h(a))_{a \in E(Y)}$
with $\phi^h(a) \in G_{t(h \cdot a)}$, satisfying the definition of morphism above (Definition~\ref{d:morphism}).

We say that the $H$--action is \emph{by simple morphisms} if each $\Phi^h$ is simple, that is, if each $\phi^h(a) \in G_{t(h \cdot a)}$ is the trivial element.  Explicitly, for each $a \in E(Y)$ and each $h \in H$, the following diagram commutes. \[\xymatrix{
G_{i(a)}   \ar[d]^-{\phi^h_{i(a)}} \ar[rrr]^{\psi_a} & & & G_{t(a)} \ar[d]^-{\phi^h_{t(a)}}
\\
G_{h \cdot i(a)}  \ar[rrr]^{\psi_{h \cdot a}} & & & G_{h \cdot t(a)}
}\]   We note also that the composition of simple morphisms $\Phi^{h'} \circ \Phi^{h}$ is the simple morphism
$\Phi^{h'h}$ with local maps \begin{equation}\label{e:composition}\phi^{h' h}_\s = \phi^{h'}_{h\cdot\s}\circ \phi^{h}_\s.\end{equation}

Finally we recall the definition of a covering of complexes of groups.

\begin{definition}\label{d:covering} A morphism $\Phi:G(Y) \to H(Z)$ over a nondegenerate morphism of
scwols $f:Y\to Z$ (see Definition~\ref{d:morphism_scwols} above) is a
\emph{covering of complexes of groups} if further: \begin{enumerate}\item each $\phi_\sigma$ is
injective; and \item \label{i:covbijection} for each $\sigma \in V(Y)$ and $b \in E(Z)$ such that
$t(b) = f(\sigma)$, the map on cosets \[  \Phi_{\s/b}:\left(\coprod_{\substack{a \in f^{-1}(b)\\ t(a)=\sigma}} G_\sigma /
\psi_a(G_{i(a)})\right) \to H_{f(\sigma)} / \theta_b(H_{i(b)})\] induced by $g \mapsto
\phi_\sigma(g)\phi(a)$ is a bijection.\end{enumerate}\end{definition}

\subsection{The induced complex of groups and its properties}\label{ss:induced_cogs}

Suppose that a group $H$ acts by simple morphisms on a simple complex of groups $G(Y)=(G_\s,\psi_a)$.  In
this section we construct the complex of groups $H(Z)$ induced by this action, prove that $H(Z)$ is
well-defined up to isomorphism of complexes of groups and discuss faithfulness.

Let $Z$ be the quotient scwol $Z = H \bs Y$ and let $p:Y \to Z$ be the natural projection.  For each
vertex $\tau \in V(Z)$ choose a representative $\overline\tau \in V(Y)$ such that $p(\overline\tau) =
\tau$.  Let  $\Stab_H(\overline\tau)$ be the subgroup of $H$ fixing $\overline\tau$ and let $G_{\overline\tau}$ be the local
group of $G(Y)$ at $\overline\tau$.  Since the $H$--action is by simple morphisms, by Equation~\eqref{e:composition} above there is a group homomorphism $\zeta:\Stab_H(\overline\tau) \to \Aut(G_{\overline\tau})$ given by $\zeta(h) = \phi^h_{\overline\tau}$.  For each $\tau \in V(Z)$ we then define the local group $H_\tau$ to be the corresponding semidirect product of $G_{\overline\tau}$ by $\Stab_H(\overline\tau)$, that is, \[H_\tau := G_{\overline\tau} \rtimes_\zeta \Stab_H(\overline\tau) = G_{\overline\tau} \rtimes \Stab_H(\overline\tau).\]

For each edge $a \in E(Z)$ with $i(a) = \tau$ there is, since $H$ acts on $Y$ in the sense of
Definition~\ref{d:action_on_scwol} above, a unique edge $\overline{a} \in E(Y)$ such that
$p(\overline{a})=a$ and $i(\overline{a}) = \overline{i(a)} = \overline\tau$. For each $a \in E(Z)$
choose an element $h_a \in H$ such that $h_a\cdot t(\overline{a}) = \overline{t(a)}$.

\begin{lemma}\label{l:monom} Let $g \in G_{i(\overline{a})} = G_{\overline{i(a)}}$ and $h \in
\Stab_H\left(\overline{i(a)}\right)$.  Then the map \[ \theta_a: (g,h) \mapsto
(\phi^{h_a}_{t(\overline{a})} \circ \psi_{\overline{a}} (g), h_a h h_a^{-1}) \] is a monomorphism
$H_{i(a)} \to H_{t(a)}$. \end{lemma}

\begin{proof} We will show that $\theta_a$ is a group homomorphism.  Since $\phi^{h_a}_{t(\overline{a})}$, $\psi_{\overline{a}}$ and the conjugation $h \mapsto h_a h h_a^{-1}$ are all injective, the conclusion will then follow.

Let $g,g' \in G_{\overline{i(a)}}$ and $h,h' \in \Stab_H(\overline{i(a)})$.  Note that since $h$ and $h'$ fix $\overline{i(a)} = i(\overline{a})$, they fix the edge $\overline{a}$ and hence fix the vertex $t(\overline{a})$ as well.  We have
\begin{eqnarray*}
\theta_a((g,h)(g,h')) & = & \theta_a(g\phi^h_{\overline{i(a)}}(g'), hh') \\
& = & (\phi^{h_a}_{t(\overline{a})}\circ \psi_{\overline{a}}(g\phi^h_{\overline{i(a)}}(g')),h_ahh'h_a^{-1})
\end{eqnarray*}
while
\begin{eqnarray*}
\theta_a(g,h)\theta_a(g',h') & = & (\phi^{h_a}_{t(\overline{a})}\circ \psi_{\overline{a}}(g),h_ahh_a^{-1})(\phi^{h_a}_{t(\overline{a})}\circ \psi_{\overline{a}}(g'),h_ah'h_a^{-1})\\
& = & (\phi^{h_a}_{t(\overline{a})}\circ \psi_{\overline{a}}(g)\phi^{h_ahh_a^{-1}}_{\overline{t(a)}}\circ \phi^{h_a}_{t(\overline{a})}\circ \psi_{\overline{a}}(g'),h_ahh'h_a^{-1}).
\end{eqnarray*}
After applying Equation~\eqref{e:composition} above to the map $\phi^{h_a h h_a^{-1}}$, and some cancellations, it remains to show that
\[\psi_{\overline{a}} \circ \phi^h_{\overline{i(a)}}(g') = \phi^h_{t(\overline{a})} \circ \psi_{\overline{a}}(g').\]
This follows from the fact that $\Phi^h$ is a simple morphism with $h\cdot \overline{a} = \overline{a}$.
\end{proof}

To complete the construction of $H(Z)$, for each composable pair of edges $(a,b)$ in $E(Z)$, define
\[h_{a,b} = h_a h_b h_{ab}^{-1}.\]
One checks that $h_{a,b} \in \Stab_H(\overline{t(a)})$ hence $(1,h_{a,b}) \in H_{t(a)}$.  By abuse of notation we write
$h_{a,b}$ for $(1, h_{a,b})$.

\begin{proposition} The datum $H(Z) = (H_\s, \theta_a, h_{a,b})$ is a complex of groups.
\end{proposition}

\begin{proof} Given Lemma~\ref{l:monom} above, it remains to show that for each pair of composable edges
$(a,b)$ in $E(Z)$, \begin{equation}\label{e:thetas} \Ad(h_{a,b}) \circ \theta_{ab} = \theta_a
\circ \theta_b,\end{equation} and that the cocycle condition holds.  Let $(g,h) \in H_{i(b)}=G_{\overline{i(b)}} \rtimes
\Stab_H(\overline{i(b)})$.  We compute \begin{eqnarray*}\Ad(h_{a,b}) \circ \theta_{ab} (g,h) & = &
(\phi^{h_{a,b}}_{\overline{t(ab)}}\circ \phi^{h_{ab}}_{t(\overline{ab})} \circ \psi_{\overline{ab}}(g),h_{a,b}h_{ab} h
h_{ab}^{-1}h_{a,b}^{-1}) \end{eqnarray*} while
\begin{eqnarray*}\theta_a \circ \theta_b (g,h) & = & (\phi^{h_a}_{t(\overline{a})} \circ
\psi_{\overline{a}} \circ  \phi^{h_b}_{t(\overline{b})} \circ \psi_{\overline{b}}(g), h_a h_b h
h_b^{-1} h_a^{-1}).\end{eqnarray*}  By definition of $h_{a,b}$ it remains to show equality in the first component.

By Equation~\eqref{e:composition} and the definition of $h_{a,b}$,
\[ \phi^{h_{a,b}}_{\overline{t(ab)}} = \phi^{h_a}_{t(\overline{a})} \circ \phi^{h_b}_{t(\overline{ab})} \circ \phi^{h_{ab}^{-1}}_{\overline{t(ab)}}. \]
Hence it suffices to prove
\begin{equation}\label{e:useful}\phi^{h_b}_{t(\overline{ab})}
\circ \psi_{\overline{ab}} = \psi_{\overline{a}} \circ \phi^{h_b}_{t(\overline{b})} \circ \psi_{\overline{b}}. \end{equation}
Since $G(Y)$ is a simple complex of groups, and $\overline{ab}$ is the composition of the edges
$h_b^{-1}.\overline{a}$ and $\overline{b}$, we have  \[ \psi_{\overline{ab}} = \psi_{h_b^{-1}
\overline{a}} \circ \psi_{\overline{b}}.\] Applying this, and the fact that $\phi^{h_b}_{t(\overline{b})}$ is a simple morphism on the edge $h_b^{-1} \overline{a}$, we have \[ \phi^{h_b}_{t(\overline{ab})} \circ \psi_{\overline{ab}}  = \phi^{h_b}_{t(\overline{ab})} \circ \psi_{h_b^{-1} \overline{a}} \circ \psi_{\overline{b}} =
 \psi_{\overline{a}} \circ \phi^{h_b}_{t(\overline{b})} \circ \psi_{\overline{b}}.\]
Hence Equation~\eqref{e:useful} holds.

The cococycle condition follows from the definition of $h_{a,b}$.  We conclude that $H(Z)$ is a complex of groups.
\end{proof}

We now have a complex of groups $H(Z)$ induced by the action of $H$ on $G(Y)$.  This construction
depended on choices of lifts $\overline\tau$ and of elements $h_a \in H$.  We next show (in a generalisation of Section~III.$\mathcal{C}$.2.9(2) of~\cite{BH}) that:

\begin{lemma}\label{l:well_defined} The complex of groups $H(Z)$ is well-defined up to isomorphism of complexes of groups.
\end{lemma}

\begin{proof} Suppose we made a different choice of lifts $\overline\tau'$ and elements $h_a'$,
resulting in a complex of groups $H'(Z) = (H'_\tau, \theta'_a, h'_{a,b})$.  An isomorphism
$\Lambda = (\lambda_\s, \lambda(a))$ from $H(Z)$ to $H'(Z)$ over the identity map $Z \to Z$ is
constructed as follows.  For each $\tau \in V(Z)$, choose an element $k_\tau \in H$ such that
$k_\tau \cdot \overline\tau = \overline\tau'$, and define a group isomorphism $\lambda_\tau:H_\tau \to
H'_\tau$ by  \[\lambda_\tau(g,h) = (\phi^{k_\tau}_{\overline\tau}(g), k_\tau h k_\tau^{-1}).\] For
each $a \in E(Z)$, define $\lambda(a) = (1, k_{t(a)} h_a k^{-1}_{i(a)} {h'_a}^{-1})$.  Note that by ~III.$\mathcal{C}$.2.9(2)
of~\cite{BH}, $\lambda(a) \in H'_{t(a)}$.

The verification that $\Lambda = (\lambda_\s,\lambda(a))$ is an isomorphism of complexes of groups is straightforward.\end{proof}

We remind the reader that faithfulness of a complex of groups is defined in the final paragraph of Section~\ref{ss:complexes_of_groups} above.

\begin{lemma}\label{l:faithful} If $G(Y)$ is faithful and the $H$--action on $Y$ is faithful then $H(Z)$ is faithful.
\end{lemma}

\begin{proof} This follows from the construction of $H(Z)$, and the characterisation of faithful
complexes of groups in Proposition 38 of~\cite{LT}.
\end{proof}

\subsection{The induced covering}\label{ss:induced_covering}

Suppose $H$ acts by simple morphisms on a simple complex of groups $G(Y)$, inducing a complex of groups $H(Z)$ as in Section~\ref{ss:induced_cogs} above.  In this section we construct a covering of complexes of groups $\Lambda:G(Y) \to H(Z)$ over the quotient map $p:Y \to Z$.

For $\s \in V(Y)$, the local maps $\lambda_\s:G_\s \to H_{p(\s)}$ are defined as follows.  Recall
that for each vertex $\tau \in V(Z)$ we chose a lift $\overline\tau \in V(Y)$.  Now for each $\s \in V(Y)$, we choose $k_\s \in H$ such that $k_\s \cdot \s = \overline{p(\s)}$.  Hence $\phi^{k_\s}_\s$ is an isomorphism $G_\s \to G_{\overline{p(\s)}}$.  The local map $\lambda_\s:G_\s \to
H_{p(\s)}$ is then defined by \[\lambda_\s: g \mapsto (\phi^{k_\s}_\s(g), 1).\]  Note
that each $\lambda_\s$ is injective.

For each edge $a \in E(Y)$, define \[\lambda(a) = (1, k_{t(a)}k_{i(a)}^{-1} h_b^{-1})\] where
$p(a) = b \in E(Z)$.  Note that, since $H$ acts on $Y$ in the sense of
Definition~\ref{d:action_on_scwol} above, we have $k_{i(a)}\cdot a = \overline{b}$
hence $k_{t(a)}k_{i(a)}^{-1} h_b^{-1}$ fixes $\overline{t(b)}$. Thus $\lambda(a) \in
H_{t(b)}$ as required.

\begin{proposition}\label{p:covering} The map $\Lambda = (\lambda_\s,\lambda(a))$ is a covering of complexes of groups. \end{proposition}

\begin{proof} It may be checked that $\Lambda$ is a morphism of complexes of groups.  As noted, each
of the local maps $\lambda_\s$ is injective. It remains to show that for each $\sigma \in V(Y)$ and
$b\in E(Z)$ such that $t(b) = p(\sigma)=\tau$, the map on cosets \[  \Lambda_{\s/b}:\left(\coprod_{\substack{a \in p^{-1}(b)\\
t(a)=\sigma}} G_\sigma / \psi_a(G_{i(a)})\right) \to H_{\tau} / \theta_b(H_{i(b)})\] induced by $g
\mapsto \lambda_\s(g)\lambda(a)=(\phi^{k_\s}_\s(g), k_\s k_{i(a)}^{-1} h_b^{-1})$ is a bijection.

We first show that $\Lambda_{\s/b}$ is injective.  Suppose $a$ and $a'$ are in $p^{-1}(b)$ with
$t(a)=t(a')=\s$, and suppose $g, g' \in G_\s$ with $g$ representing a coset of
$\psi_{a}(G_{i(a)})$ in $G_\s$ and $g'$ a coset of $\psi_{a'}(G_{i(a')})$ in $G_\s$.  Assume that
$\lambda_\s(g)\lambda(a)$ and $\lambda_\s(g')\lambda(a')$ belong to the same coset of
$\theta_b(H_{i(b)})$ in $H_\tau$.

Looking at the second component of the semidirect product $H_\tau$, it follows from the definition of
$\theta_b$ (Lemma~\ref{l:monom} above) that for some $h \in \Stab_H({\overline{i(b)}})$, \[ k_\s
k_{i(a)}^{-1} h_b^{-1}  =  \left(k_\s k_{i(a')}^{-1} h_b^{-1}\right) \left( h_b h h_b^{-1}\right)
\\   = k_\s k_{i(a')}^{-1} h h_b^{-1}.\] Thus $k_{i(a')}k_{i(a)}^{-1}=h$ fixes $\overline{i(b)}$.
Hence $k_{i(a)}^{-1}k_{i(a')}$ fixes $k_{i(a)}^{-1}\overline{i(b)} = i(a)$, and so
$k_{i(a)}^{-1}k_{i(a')}$ fixes $a$.  Thus $k_{i(a')} \cdot a = k_{i(a)} \cdot a = \overline{b} = k_{i(a')} \cdot a'$,
hence $a = a'$.

Looking now at the first component of $\lambda_\s(g)\lambda(a)$ and
$\lambda_\s(g')\lambda(a')=\lambda_\s(g')\lambda(a)$ in the semidirect product $H_\tau$, by
definition of $\theta_b$, for some $x \in G_{\overline{i(b)}}$ we have \begin{eqnarray*}
\phi^{k_\s}_\s(g)  & = & \phi^{k_\s}_\s(g')\phi^{k_\s k_{i(a)}^{-1} h_{b}^{-1}}_{\overline{t(b)}}
\circ \phi^{h_b}_{t(\overline{b})} \circ \psi_{\overline{b}}(x)  \\ & = &
\phi^{k_\s}_\s(g')\phi^{k_\s}_{\s} \circ \phi^{k_{i(a)}^{-1}}_{t(\overline{b})} \circ
\psi_{\overline{b}}(x).\end{eqnarray*} Since $\phi^{k_\s}_{\s}$ is an isomorphism, and $k_{i(a)}^{-1}\cdot\overline{b} = a$, this implies \[(g')^{-1}g = \phi^{k_{i(a)}^{-1}}_{t(\overline{b})}
\circ \psi_{\overline{b}}(x) = \psi_a \circ \phi^{k_{i(a)}^{-1}}_{\overline{i(b)}}(x) \in
\psi_a(G_{i(a)})\] as required. Thus the map $\Lambda_{\s/b}$ is injective.

To show that $\Lambda_{\s/b}$ is surjective, let $g \in G_{\overline{\tau}}$ and $h \in \Stab_H(\overline\tau)$, so
that $(g,h) \in H_\tau$.  Let $a$ be the unique edge of $Y$ with $t(a) =\s$ and such that $k_\s\cdot a = h h_b
\overline{b}$.  Let $g'$ be the unique element of $G_\s$ such that $\phi^{k_\s}_\s(g') = g \in
G_{\overline\tau}$. We claim that $\lambda_\s(g')\lambda(a)$ lies in the same coset as $(g,h)$.  Now
\[\lambda_\s(g')\lambda(a) = (\phi^{k_\s}_\s(g'), k_\s k_{i(a)}^{-1} h_b^{-1})=(g,k_\s k_{i(a)}^{-1} h_b^{-1})\]
so it suffices to show that $k_\s k_{i(a)}^{-1} h_b^{-1} \in h h_b \Stab_H(\overline{i(b)}) h_b^{-1}$.
Equivalently, we wish to show that $h_b^{-1} h^{-1} k_\s k_{i(a)}^{-1}$ fixes $\overline{i(b)}$.  We have
$k_{i(a)}\cdot i(a) = \overline{i(b)}$ by definition, and the result follows by our choice
of $a$.  Thus $\Lambda_{\s/b}$ is surjective.

Hence $\Lambda$ is a covering of complexes of groups.\end{proof}

\subsection{The fundamental group}\label{ss:fund_group}

Suppose $H$ acts by simple morphisms on a simple complex of groups $G(Y)$, inducing a complex of groups $H(Z)$ as in Section~\ref{ss:induced_cogs} above.  In this section we establish the short exact sequence of  Theorem~\ref{t:group_action} above, and provide sufficient conditions for the fundamental group of $H(Z)$ to be the semidirect product of the fundamental group of $G(Y)$ by $H$.

Fix $\s_0$ a vertex of $Y$ and let $p:Y \to Z$ be the natural projection.  We refer the reader to Section~III.$\mathcal{C}$.3 of~\cite{BH} for the definition of the \emph{fundamental group of G(Y) at $\s_0$}, denoted $\pi_1(G(Y),\s_0)$.  We will use notation and results from that section in the following proof.  Let $\pi_1(H(Z),p(\s_0))$ be the fundamental group of $H(Z)$ at $p(\s_0)$.

\begin{proposition}\label{p:SES}  There is a natural short exact sequence
\[1 \to \pi_1(G(Y),\s_0) \to \pi_1(H(Z),p(\s_0)) \to H \to 1.\]
\end{proposition}

\begin{proof}
To obtain a monomorphism $\pi_1(G(Y),\s_0) \to \pi_1(H(Z),p(\s_0))$, we use the morphism of complexes of groups $\Lambda:G(Y) \to H(Z)$ defined in Section~\ref{ss:induced_covering} above.  By Proposition~III.$\mathcal{C}$.3.6 of~\cite{BH}, $\Lambda$ induces a natural homomorphism
\[\pi_1(\Lambda,\s_0):\pi_1(G(Y),\s_0) \to \pi_1(H(Z),p(\s_0)).\] Since
$\Lambda$ is a covering (Proposition~\ref{p:covering} above), Theorem~\ref{t:coverings} above implies
that this map $\pi_1(\Lambda,\s_0)$ is in fact injective.

We next define a surjection $\pi_1(H(Z),p(\s_0)) \to H$.  The group $H$ may be regarded as a complex of groups over a single vertex.  There is then a canonical morphism of complexes of groups $\Phi:H(Z) \to H$, defined as follows.  Recall that for each $\tau \in V(Z)$, the local group $H_\tau$ is given by $H_\tau = G_{\overline\tau} \rtimes \Stab_H(\overline\tau)$.  The local map $\phi_\tau:H_\tau \to H$ in the morphism $\Phi$ is defined to be projection to the second factor $\Stab_H(\overline{\tau}) \leq H$.  For each edge $b$ of $Z$, we define $\phi(b) = h_b$.  It may then be checked that $\Phi$ is a morphism.

By Proposition~III.$\mathcal{C}$.3.6 of~\cite{BH}, the morphism $\Phi$ induces a homomorphism of fundamental groups
\[\pi_1(\Phi,p(\s_0)): \pi_1(H(Z),p(\s_0)) \to H.\]
By~III.$\mathcal{C}$.3.14 and Corollary~III.$\mathcal{C}$.3.15 of~\cite{BH}, if $G(Y)$ were a complex of trivial groups, this map would be surjective.  Since the image of $\pi_1(\Phi,p(\s_0))$ does not in fact depend on the local groups of $G(Y)$, we have that in all cases, $\pi_1(\Phi,p(\s_0))$ is surjective, as required.

It follows from definitions that the image of the monomorphism $\pi_1(\Lambda,\s_0)$ is the kernel of the surjection $\pi_1(\Phi,p(\s_0))$.  Hence the sequence above is exact.\end{proof}

\begin{corollary}\label{p:fund_gp_splits}
If $H$ fixes a vertex of $Y$,
\[ \pi_1(H(Z),p(\s_0)) \cong \pi_1(G(Y),\s_0) \rtimes H.\]
\end{corollary}

\begin{proof}
Suppose that $H$ fixes the vertex $\sigma$ of $Y$.  We will construct a section $\iota:H \to \pi_1(H(Z),p(\s_0))$ for the surjective homomorphism $\pi_1(\Phi,p(\s_0)): \pi_1(H(Z), p(\s_0)) \to H$ given in the proof of Proposition~\ref{p:SES} above.

The vertex $\sigma$ is the unique lift $\overline\tau$ of a vertex $p(\sigma) = \tau \in Z$.  Hence
\[H_\tau = G_{\overline\tau} \rtimes \Stab_H(\overline\tau) = G_\s \rtimes H.\]
By definition of the surjection $\pi_1(\Phi,p(\s_0)): \pi_1(H(Z), p(\s_0)) \to H$, a section $\iota:H \to \pi_1(H(Z),p(\s_0))$ is then given by the inclusion $H \to H_\tau$.
 \end{proof}

This completes the proof of Theorem~\ref{t:group_action}.

\section{Proof of the Main Theorem}\label{s:proof}

We now prove the Main Theorem and Corollary~\ref{c:infinite_generation}, stated in the introduction.  Throughout this section, we adopt
the notation of the Main Theorem, and assume that the vertices $s_1$ and $s_2$ of the nerve $L$, and the
elements $\alpha_1$ and $\alpha_2$ of the group $A$ of label-preserving automorphisms of $L$, satisfy
Conditions~\eqref{c:fix}--\eqref{c:halvable} of its statement.  In Section~\ref{ss:underlying} we
introduce notation, and construct a family of finite polyhedral complexes $Y_n$, for $n \geq 1$, and an
infinite polyhedral complex $Y_\infty$.  We then in Section~\ref{ss:GYn} construct complexes of groups
$G(Y_n)$ and $G(Y_\infty)$ over these spaces, and show that there are coverings of complexes of groups
$G(Y_n) \to G(Y_1)$ and $G(Y_\infty) \to G(Y_1)$.  In Section~\ref{ss:Hn} we define the action of a
finite group $H_n$ on $Y_n$, and of an infinite group $H_\infty$ on $Y_\infty$, and then in
Section~\ref{ss:action} we show that these actions extend to actions on the complexes of groups
$G(Y_n)$ and $G(Y_\infty)$.  In Section~\ref{ss:conclusion} we combine these results with
Theorem~\ref{t:group_action}  above to complete the proof of the Main Theorem.
Corollary~\ref{c:infinite_generation} is proved in Section~\ref{ss:corollary_proof}.

\subsection{The spaces $Y_n$ and $Y_\infty$}\label{ss:underlying}

In this section we construct a family of finite polyhedral complexes $Y_n$ and an infinite polyhedral
complex $Y_\infty$.

We first set up some notation.
For $i = 1,2$, let $q_i \geq 2$ be the order of $\alpha_i$.  It will be convenient to put, for all $k \geq 0$, $s_{2k+1} = s_1$ and $s_{2k+2} = s_2$, and similarly $\alpha_{2k+1} = \alpha_1$, $\alpha_{2k+2}=\alpha_2$, $q_{2k+1}=q_1$ and $q_{2k+2} = q_2$.  Conditions~\eqref{c:fix}--\eqref{c:halvable} of the Main Theorem then become:
\begin{enumerate}
\item for all $n \geq 1$, $\alpha_n$ fixes the star of $s_{n+1}$ in $L$;
\item for all $n \geq 1$, the subgroup $\langle \alpha_n \rangle$ of $A$ acts freely on the $\langle \alpha_{n} \rangle$--orbit of $s_n$, in particular $\alpha_n(s_{n}) \neq s_{n}$;
\item for all $n \geq 1$, and all $t_n \neq s_n$ such that $t_n$ is in the $\langle \alpha_{n} \rangle$--orbit of $s_n$, $m_{s_{n}t_n} = \infty$; and
\item for all $n \geq 1$, all spherical special subgroups of $W$ which contain $s_n$ are halvable along $s_n$.
\end{enumerate}

We now use the sequences $\{s_n\}$ and $\{ \alpha_n\}$ to define certain elements and subsets of $W$.
Let $w_1$ be the trivial element of $W$ and for $n \geq 2$ let $w_n$ be the product
\[ w_{n} = s_1 s_2 \cdots s_{n-1} \in W.\]
Denote by $W_{n,n}$ the one-element set $\{ w_{n} \}$.  For $n \geq 2$, and $1 \leq k < n$, in order to simplify notation, write $\alpha^{j_{n-1},\ldots,j_k}$ for the composition of automorphisms
\[ \alpha^{j_{n-1},\ldots,j_k}= \alpha_{n-1}^{j_{n-1}}\cdots \alpha_k^{j_k} \]
where $0 \leq j_i < q_i$ for $k \leq i < n$.   Let $w_{j_{n-1},\ldots,j_k}$ be the element of $W$:
\begin{equation}\label{e:wk} w_{j_{n-1},\ldots,j_k} = w_n \alpha^{j_{n-1}}(s_{n-1})\alpha^{j_{n-1}, j_{n-2}}(s_{n-2})\cdots\alpha^{j_{n-1}, \ldots ,j_{k+1}}(s_{k+1})\alpha^{j_{n-1},\ldots, j_k}(s_k). \end{equation}
Now for $n \geq 2$ and $1 \leq k < n$, define
\[
W_{k,n}  =   \{  w_{j_{n-1},\ldots,j_k} \in W \mid \mbox{$0 \leq j_i < q_i$ for $k \leq i < n$} \}.
\]
Note that if $j_{n-1} =0$ then $w_{j_{n-1},\ldots,j_k} \in W_{k,n-1}$.

\example Let $(W,S)$ be the Coxeter system in Example 1 of Section~\ref{ss:Davis_complexes} above, with
nerve $L$ shown in Figure~\ref{f:nerve} above.  For $i = 1,2$, let $\alpha_i \in A$ be the automorphism
of $L$ which fixes the star of $s_{3-i}$ in $L$ and interchanges $s_i$ and $s_3$.  Then if $m$ and $m'$
are both even, the Main Theorem applies to this example. (If $T = \{ s\}$ then $W_T$ is halvable along
$s$ with $\half_s(W_T)$ the trivial group.  If $T=\{ s,t \}$ then $W_T$ is the dihedral group of order
$2m_{st}$, and $W_T$ is halvable along $s$ if and only if $m_{st}$ is even, in which case $\half_s(W_T)$
is the dihedral group of order $m_{st}$.)  Note that $q_1 = q_2 = 2$, and so, for instance,
\begin{eqnarray*} W_{1,3} & = & \{ 1, s_1\alpha_1(s_1), s_1s_2\alpha_2(s_2)\alpha_2(s_1),
s_1s_2\alpha_2(s_2)\alpha_2\alpha_1(s_1)\} \\ W_{2,3} & = & \{ s_1, s_1s_2\alpha_2(s_2) \} \\ W_{3,3} &
= & \{ s_1 s_2 \}.\end{eqnarray*}

The following lemma establishes key properties of the sets $W_{k,n}$.

\begin{lemma}\label{l:disjoint} For all $n \geq 1$: \begin{enumerate}\item the sets $W_{1,n}$,
$W_{2,n}$, \ldots, $W_{n,n}$ are pairwise disjoint; and \item for all $1 \leq k < n$, if
\[w_{j_{n-1},\ldots,j_k}=w_{j'_{n-1},\ldots,j'_k}\] (where $0 \leq j_i < q_i$ for $k \leq i < n$) then
$j_k = j'_k$, $j_{k+1}=j'_{k+1}$, \ldots, and $j_{n-1}=j'_{n-1}$.\end{enumerate} \end{lemma}

\begin{proof} Given $1 \leq k \leq k' <n$, with $0 \leq j_i < q_i$ for $k \leq i < n$ and $0 \leq j'_i <
q_i$ for $k' \leq i < n$, suppose \begin{equation}\label{e:equal_ws} w_{j_{n-1},\ldots,j_k} =
w_{j'_{n-1},\ldots,j'_{k'}}.\end{equation} Then \[\alpha^{j_{n-1}}(s_{n-1})\alpha^{j_{n-1},
j_{n-2}}(s_{n-2})\cdots\alpha^{j_{n-1}, \ldots, j_{k'},\ldots,
,j_{k+1}}(s_{k+1})\alpha^{j_{n-1},\ldots,j_{k'},\ldots, j_k}(s_k)
\]\[=\alpha^{j'_{n-1}}(s_{n-1})\alpha^{j'_{n-1}, j'_{n-2}}(s_{n-2})\cdots\alpha^{j'_{n-1}, \ldots
,j'_{k'+1}}(s_{k'+1})\alpha^{j'_{n-1},\ldots, j'_{k'}}(s_{k'}).\] By Condition~\eqref{c:fix} above, for each
$k \leq i < n$, the automorphism $\alpha_i$ fixes $s_{i+1}$, thus \begin{eqnarray*} \alpha^{j_{n-1}, \ldots
,j_{i+1}}(s_{i+1})\alpha^{j_{n-1},\ldots, j_i}(s_i) & = & \alpha^{j_{n-1}, \ldots
,j_{i+1},j_i}(s_{i+1})\alpha^{j_{n-1},\ldots, j_i}(s_i) \\ & = & \alpha^{j_{n-1},\ldots, j_i}(s_{i+1}s_i).
\end{eqnarray*} Also since $\alpha_i$ fixes the star of $s_{i+1}$ but $\alpha_i(s_i) \neq s_i$, we
have $m_{s_{i+1}s_i} = \infty$.  Since $\alpha^{j_{n-1},\ldots, j_i}$ is a label-preserving automorphism, it
follows that the product of the two generators \[ \alpha^{j_{n-1}, \ldots
,j_{i+1}}(s_{i+1})\alpha^{j_{n-1},\ldots, j_i}(s_i) \] has infinite order, for each $k \leq i < n$.
Similarly for each $k'\leq i < n$.  Thus the only way for Equation~\eqref{e:equal_ws} to hold is if $k =
k'$, and for each $k \leq i < n$, $\alpha_i^{j_i}(s_i) = \alpha_i^{j'_i}(s_i)$.  Since $\langle \alpha_i \rangle$ acts freely on the $\langle \alpha_i \rangle$--orbit of $s_i$ and we specified $0 \leq j_i < q_i$, the result follows. \end{proof}

 For $n \geq 1$, and $1 \leq k \leq n$, define $Y_{k,n}$ to be the set of chambers
\[ Y_{k,n} := \{ w K \mid w \in W_{k,n} \}. \]
Recall that we are writing $wK$ for the pair $(w,K)$.  By Lemma~\ref{l:disjoint} above, for fixed $n$, the sets $Y_{1,n},\ldots,Y_{n,n}$ are pairwise disjoint.  We now define $Y_n$ to be the polyhedral complex obtained by ``gluing together" the chambers in $Y_{1,n},\ldots,Y_{n,n}$, using the same relation $\sim$ as in the Davis complex $\Sigma$ for $(W,S)$.  More precisely,
\[Y_n := \left(\coprod_{k=1}^n Y_{k,n}\right) / \sim\]
where, for $x,x' \in K$, we have $(w,x) \sim (w',x')$ if and only if $x = x'$ and $w^{-1}w' \in W_{S(x)}$.  Note that $Y_1 = Y_{1,1} =  K$.  To define $Y_\infty$, for each $k \geq 1$, noting that $W_{k,n}$ is only defined for $1 \leq k \leq n$, put
\[ W_{k,\infty} := \bigcup_{n=k}^\infty W_{k,n}.\]
Then $Y_{k,\infty}$ is the set of chambers
\[Y_{k,\infty} := \{ wK \mid w \in W_{k,\infty}\}.\] Similarly to the finite case, the sets $Y_{1,\infty},Y_{2,\infty},\ldots$ are pairwise disjoint, and we define
\[Y_\infty = \left(\coprod_{k=1}^\infty Y_{k,\infty}\right) / \sim\]
for the same relation $\sim$.  Note that there are natural strict inclusions as subcomplexes
\[Y_1 \subset Y_2 \subset \cdots \subset Y_n \subset \cdots Y_\infty.\]
(In fact, $Y_n$ and $Y_\infty$ are subcomplexes of the Davis complex $\Sigma$, but we will not adopt this point of view.)
We define a mirror of $Y_n$ or $Y_\infty$ to be an \emph{interior mirror} if it is contained in more than one chamber.

\example Let $(W,S)$, $\alpha_1$ and $\alpha_2$ be as in the previous example of this section.  To indicate the construction of $Y_n$ and $Y_\infty$ in this case, Figure~\ref{f:Y4} below
depicts the dual graph for $Y_4$, that is, the graph with vertices the chambers of $Y_4$, and edges joining adjacent chambers.  The edges are labelled with the type of the corresponding interior mirror.  Figure~\ref{f:Yinfty} sketches the dual graph for $Y_\infty$.

\begin{figure}[ht]
\begin{center}
\includegraphics[scale=1, viewport=150 450 450 750, clip]{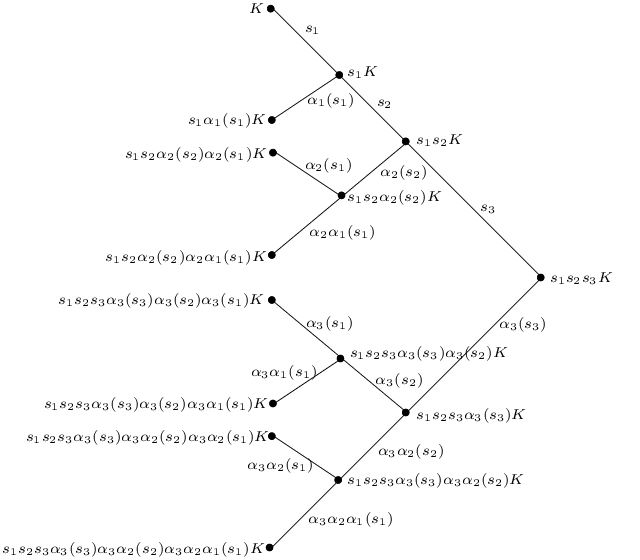}
\end{center}
\caption{Dual graph for $Y_4$, with vertices and edges labelled}
\label{f:Y4}
\end{figure}

\begin{figure}[ht]
\begin{center}
\includegraphics{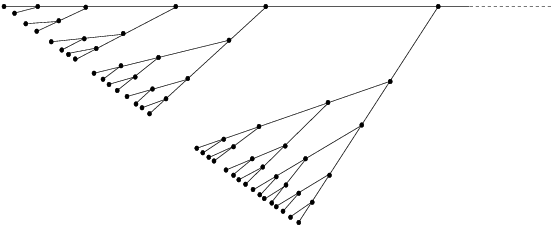}
\end{center}
\caption{Dual graph for $Y_\infty$}
\label{f:Yinfty}
\end{figure}

We now describe features of $Y_n$ and $Y_\infty$ which will be needed below.  The first lemma follows from the construction of $Y_n$ and $Y_\infty$ and Lemma~\ref{l:disjoint} above.

\begin{lemma}\label{l:adjacency} Let $w=w_{j_{n-1},\ldots,j_k} \in W_{k,n}$.  All of the chambers of $Y_n$ to which $wK \in Y_{k,n}$ is adjacent are described by the following.
\begin{enumerate}
\item\label{i:up} For $n \geq 1$ and $1 \leq k < n$, the chamber $wK$ is adjacent to exactly one chamber of $Y_{k+1,n}$, namely it is $\alpha^{j_{n-1},\ldots,j_k}(s_k)$--adjacent to the chamber $w_{j_{n-1},\ldots,j_{k+1}}K$ of $Y_{k+1,n}$.
\item\label{i:down} For $n \geq 2$ and $1 \leq k \leq n$, the chamber $wK$ is adjacent to exactly $q_{k-1}$ distinct chambers of $Y_{k-1,n}$, namely for each $0 \leq j_{k-1} < q_{k-1}$, the chamber $wK$ is $\alpha^{j_{n-1},\ldots,j_k,j_{k-1}}(s_{k-1})$--adjacent to the chamber $w_{j_{n-1},\ldots,j_k,j_{k-1}}K$ of $Y_{k-1,n}$.
\end{enumerate}
Similarly for $Y_\infty$.
\end{lemma}

\begin{corollary}\label{c:mirrors}
\begin{enumerate}
\item\label{i:vertex} Any vertex of $Y_n$ is contained in at most two distinct chambers of $Y_n$, and similarly for $Y_\infty$.
\item\label{i:mirrors} Any two interior mirrors of $Y_n$ or $Y_\infty$ are disjoint.
\end{enumerate}
\end{corollary}

\begin{proof} Suppose $\s$ is a vertex of $Y_n$, contained in the chamber $wK$, where $w$ is as in Lemma~\ref{l:adjacency} above.  If $\s$ is contained in more than one chamber of $Y_n$ or $Y_\infty$, then $\s$ is contained in an interior mirror $K_s$, for some $s \in S$.  By the construction of $Y_n$ and Lemma~\ref{l:adjacency} above, $s$ is either an image of $s_k$, or one of $q_{k-1}$ distinct images of $s_{k-1}$, under some element of $A$.  Suppose $s$ is in the image of $s_k$.  Condition~\eqref{c:fix} of the Main Theorem implies that $m_{s_ks_{k-1}} = \infty$.  Hence the mirror $K_s$ is disjoint from each of the $q_{k-1}$ mirrors of types the $q_{k-1}$ images of $s_{k-1}$.  Therefore the only chambers of $Y_n$ which contain $\s$ are the two chambers $wK$ and $wsK$.  Now suppose $s$ is one of the $q_{k-1}$ images of $s_{k-1}$ under some element of $A$.  Condition~\eqref{c:orbit} of the Main Theorem implies that the mirrors of types each of these images are pairwise disjoint, and so again $\s$ is contained in only two distinct chambers of $Y_n$.  Similarly, any two interior mirrors of $Y_n$ or $Y_\infty$ are disjoint.
\end{proof}

\begin{corollary}\label{l:subYn}  For all $n \geq 2$, there are $q_{n-1}$ disjoint subcomplexes of $Y_n$,
denoted $Y_{n-1}^{j_{n-1}}$ for $0 \leq j_{n-1} < q_{n-1}$, each isomorphic to $Y_{n-1}$, and with
$Y_{n-1}^0 = Y_{n-1} \subset Y_n$.   For each $0 \leq j_{n-1} < q_{n-1}$, the subcomplex $Y_{n-1}^{j_{n-1}}$
is attached to the chamber $w_{n}K=s_1 s_2 \cdots s_{n-1}K$ of $Y_n$ along its mirror of type
$\alpha^{j_{n-1}}(s_{n-1})$.  An isomorphism \[F^{j_{n-1}}: Y_{n-1} \to Y_{n-1}^{j_{n-1}}\] is given by
sending the chamber \[w_{j_{n-2},\ldots,j_k}K \in Y_{k,n-1}\] to the chamber \[w_{j_{n-1},j_{n-2},\ldots,j_k}K
\in Y_{k,n},\] and the vertex of $w_{j_{n-2},\ldots,j_k}K$ of type $T$ to the vertex of
$w_{j_{n-1},j_{n-2},\ldots,j_k}K$ of type $\alpha^{j_{n-1}}(T)$, for each spherical subset $T$ of $S$.
\end{corollary}

\begin{proof} By induction on $n$, using Lemma~\ref{l:adjacency} and Corollary~\ref{c:mirrors} above.
\end{proof}

\subsection{Complexes of groups $G(Y_n)$ and $G(Y_\infty)$}\label{ss:GYn}

We now construct complexes of groups $G(Y_n)$ over each $Y_n$, and $G(Y_\infty)$ over $Y_\infty$, and show that there are coverings $G(Y_n) \to G(Y_1)$ and $G(Y_\infty) \to G(Y_1)$.  To simplify notation, write $Y$ for $Y_n$ or $Y_\infty$.

To define the local groups of $G(Y)$, let $\s$ be a vertex of $Y$, of type $T$.  By
Corollary~\ref{c:mirrors} above, $\s$ is contained in at most two distinct chambers of $Y$. If $\s$ is
only contained in one chamber of $Y$, put $G_\s=W_T$.  If $\s$ is contained in two distinct chambers of
$Y$, then by Corollary~\ref{c:mirrors} above $\s$ is contained in a unique interior mirror $K_s$, with $s \in
T$.  By the construction of $Y$, $s$ is in the $A$--orbit of some $s_n$, $n \geq 1$.  By
Condition~\eqref{c:halvable} of the Main Theorem, it follows that the group $W_T$ is halvable along
$s$.  We define the local group at $\s$ to be $G_\s=\half_{s}(W_T)$.

The monomorphisms between local groups are defined as follows.  Let $a$ be an edge of $Y$, with $i(a)$ of type $T$ and $t(a)$ of type $T'$, so that $T \subsetneq T'$.  If both of the vertices $i(a)$ and $t(a)$ are contained in a unique chamber of $Y$, then the monomorphism $\psi_a$ along this edge is defined to be the natural inclusion $W_T \hookrightarrow W_{T'}$.  If $i(a)$ is contained in two distinct chambers, then $i(a)$ is contained in a unique interior mirror $K_s$, with $s \in T$.  Thus $s \in T'$ as well, and so $t(a)$ is also contained in the mirror $K_s$.  From the definitions of $\half_s(W_T)$ and $\half_s(W_{T'})$, it follows that there is a natural inclusion $\half_s(W_T) \hookrightarrow \half_s(W_{T'})$, and we define $\psi_a$ be this inclusion.   Finally suppose $i(a)$ is contained in a unique chamber of $Y$ but $t(a)$ is contained in two distinct chambers of $Y$.  Then for some $k \geq 1$, $i(a)$ is in a chamber of $Y_{k,n}$ (respectively, $Y_{k,\infty}$), and $t(a)$ is either in $Y_{k-1,n}$ or in $Y_{k+1,n}$ (respectively, in $Y_{k-1,\infty}$ or $Y_{k+1,\infty}$).  Moreover $t(a)$ is contained in a unique interior mirror $K_s$, with $s \in T' - T$.  If $t(a)$ is in $Y_{k-1,n}$ (respectively, $Y_{k-1,\infty}$), then we define $\psi_a$ to be the natural inclusion $W_T \hookrightarrow \half_s(W_{T'})$.  If $t(a)$ is in $Y_{k+1,n}$ (respectively, $Y_{k+1,\infty}$), then we define $\psi_a$ to be the monomorphism defined on the generators $t \in T$ of $W_T$ by $\psi_a(t) := sts \in \half_s(W_{T'})$, that is, $\psi_a = \Ad(s)$.

It is not hard to verify that for all pairs of composable edges $(a,b)$ in $Y$, $\psi_{ab} = \psi_a \circ \psi_b$.  Hence we have constructed simple complexes of groups $G(Y_n)$ and $G(Y_\infty)$ over $Y_n$ and $Y_\infty$ respectively.  Note that these complexes of groups are faithful, since by construction the local group at each vertex of type $\emptyset$ is trivial.  Note also that  $G(Y_1)$ is the same complex of groups as constructed in Section~\ref{ss:complexes_of_groups} above, which has fundamental group $W$ and universal cover $\Sigma$.

\example Let $(W,S)$, $\alpha_1$ and $\alpha_2$ be as in the examples in Section~\ref{ss:underlying}
above.  The complex of groups $G(Y_2)$ is sketched in Figure~\ref{f:GY2}.  From left to
right, the three chambers here are $K$, $s_1K$ and $s_1\alpha_1(s_1)K$.  We denote by
$D_{2m}$ the dihedral group of order $2m$, with $D_m$ the dihedral group of order $m$, and similarly for
$D_{2m'}$ and $D_{m'}$ (recall that $m$ and $m'$ are even).

\vspace{3mm}

\begin{figure}[ht]
\begin{center}
\resizebox{125mm}{!}{\epsfbox{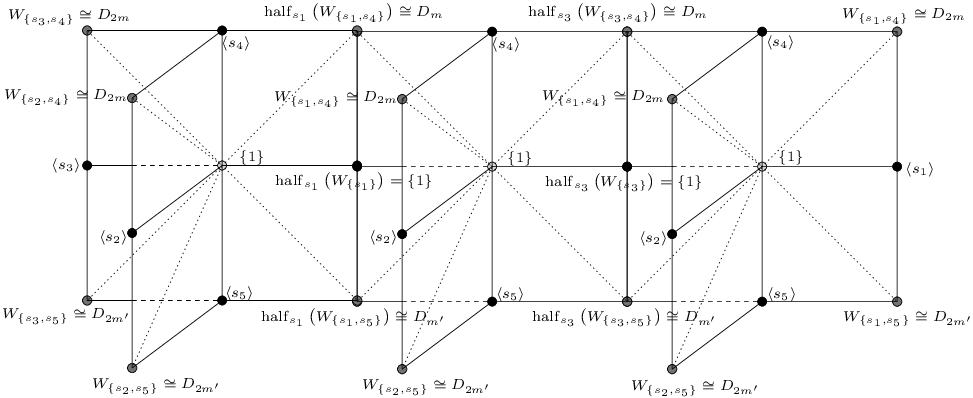}}
\end{center}
\caption{Complex of groups $G(Y_2)$}
\label{f:GY2}
\end{figure}

\begin{proposition} There are coverings of complexes of groups $G(Y_n) \to G(Y_1)$ and $G(Y_\infty) \to G(Y_1)$.
\end{proposition}

\begin{proof}  Let $f_n:Y_n \to Y_1$ and $f_\infty:Y_\infty \to Y_1$ be the maps sending each vertex of
$Y_n$ or $Y_\infty$ of type $T$ to the unique vertex of $Y_1 = K$ of type $T$.  Then by construction of
$Y_n$ and $Y_\infty$, the maps $f_n$ and $f_\infty$ are nondegenerate morphisms of scwols.  We define
coverings $\Phi_n:G(Y_n) \to G(Y_1)$ and $\Phi_\infty:G(Y_\infty) \to G(Y_1)$ over $f_n$ and $f_\infty$
respectively.  To simplify notation, write $Y$ for respectively $Y_n$ or $Y_\infty$, $f$ for respectively $f_n$ or
$f_\infty$, and $\Phi$ for respectively $\Phi_n$ or $\Phi_\infty$.

Let $\s$ be a vertex of $Y$, of type $T$.  If the local group at $\s$ is $G_\s=W_T$ then the map of
local groups $\phi_\s:G_\s \to W_T$ is the identity map.  If the local group at $\s$ is $\half_s(W_T)$,
for some $s \in T$, then $\phi_\s:\half_s(W_T) \to W_T$ is the natural inclusion as an index $2$
subgroup.  To define elements $\phi(a)$, if the monomorphism $\psi_a$ in $G(Y)$ is natural inclusion,
define $\phi(a) = 1$.  If $\psi_a$ is $\Ad(s)$, then define $\phi(a) = s$.  It is then easy to check
that, by construction, $\Phi$ is a morphism of complexes of groups.

To show that $\Phi$ is a covering of complexes of groups, we first observe that each of the local
maps $\phi_\s$ is injective.   Now fix $\sigma$ a vertex of $Y$, of type $T'$, and $b$ an edge of
$Y_1=K$ such that $t(b) = f(\s)$, with $i(b)$ of type $T$ (hence $T \subsetneq T'$).  We must show
that the map \[  \Phi_{\s/b}:\coprod_{\substack{a \in f^{-1}(b)\\ t(a)=\sigma}} G_\sigma /
\psi_a(G_{i(a)}) \to W_{T'} / W_{T}\] induced by $g \mapsto \phi_\sigma(g)\phi(a)$ is a bijection,
where $G_\s$ and $G_{i(a)}$ are the local groups of $G(Y)$.

First suppose that $\s$ is contained in a unique chamber of $Y$.  Then by construction, there is a
unique edge $a$ of $Y$ with $i(a)$ of type $T$ and $t(a) = \s$, hence a unique edge $a \in f^{-1}(b)$
with $t(a) = \s$.  Moreover, $G_\s = W_{T'}$, $G_{i(a)} = W_T$, the monomorphism $\psi_a$ is natural
inclusion hence $\phi(a) = 1$, and $\phi_\s:G_\s \to W_{T'}$ is the identity map.  Hence
$\Phi_{\s/b}$ is a bijection in this case.

Now suppose that $\s$ is contained in two distinct chambers of $Y$.  Then $\s$ is contained in a
unique interior mirror $K_s$ of $Y$, with $s \in T'$.  Assume first that $s \in T$ as well.  Then
there is a unique edge $a$ of $Y$ with $i(a)$ of type $T$ and $t(a) = \s$.  This edge is also
contained in the mirror $K_s$.  Hence there is a unique $a \in f^{-1}(b)$ with $t(a) = \s$.  By
construction, we have $G_\s = \half_s(W_{T'})$, the map $\phi_\s:G_\s \to W_{T'}$ is natural
inclusion as an index $2$ subgroup, $G_{i(a)} = \half_s(W_T)$, the map $\psi_a$ is natural inclusion,
and $\phi(a)$ trivial.  Since the index $[W_{T'}:W_T] = [\half_s(W_{T'}):\half_s(W_T)]$ is finite, it
is enough to verify that the inclusion $\half_s(W_{T'}) \to W_{T'}$ induces an injective map on
cosets \[ \half_s(W_{T'})/\half_s(W_T) \to W_{T'}/W_T.\] For this, suppose that $w,w' \in
\half_s(W_{T'})$ and that $wW_T = w'W_T$ in $W_{T'}$.  Then $w^{-1}w' \in W_T \cap \half_s(W_{T'})$.
By definitions, it follows that $w^{-1}w' \in \half_s(W_T)$, as required.

Now assume that $\s$ is contained in the interior mirror $K_s$, with $s \not \in T$.  There are then
two edges $a_1, a_2 \in f^{-1}(b)$ such that $t(a_1) = t(a_2) = \s$.  Without loss of generality,
$\psi_{a_1}$ is natural inclusion $W_T \to \half_s(W_{T'})$ and $\phi(a_1) = 1$, while $\psi_{a_2}(g)
= sgs$ with $\phi(a_2) = s$.  Since the index $[\half_s(W_{T'}):W_T] =
\frac{1}{2}[W_{T'}:W_T]$ is finite, it is enough to show that the map on cosets $\Phi_{\s/b}$ is
surjective.  Let $w \in W_{T'}$.  If $w \in \half_s(W_{T'})\leq W_{T'}$, then the image of the coset
$w\psi_{a_1}(G_{i(a_1)}) = w W_T$ in $G_\s$ is the coset $wW_T$ in $W_{T'}$.  If $w \not \in
\half_s(W_{T'})$, then since $\half_s(W_{T'})$ has index $2$ in $W_{T'}$, and $s \not\in
\half_s(W_{T'})$, there is a $w' \in \half_s(W_{T'}) \leq W_{T'}$ such that $w = w's$.  The image of
the coset $w'\psi_{a_2}(G_{i(a_2)}) = w'(sW_Ts)$ in $\half_s(W_{T'})$ is then the coset
$w'\phi(a_2)W_T = w'sW_T =wW_T$ in $W_{T'}$.  Thus $\Phi_{\s/b}$ is surjective, as required.

We conclude that $\Phi$ is a covering of complexes of groups.
\end{proof}

\subsection{Group actions on $Y_n$ and $Y_\infty$}\label{ss:Hn}

In this section we construct the action of a finite group $H_n$ on $Y_n$ in the sense of Definition~\ref{d:action_on_scwol} above, and that of an infinite group $H_\infty$ on $Y_\infty$.

We first define the groups $H_n$ and $H_\infty$.  For each $n \geq 1$, let $C_{q_n}$ denote the cyclic
group of order $q_n$.  Note that $C_{q_n} \cong \langle \alpha_n \rangle$.  We define $H_1$ to be the trivial group and $H_2 = C_{q_1}$.  For $n \geq 3$, we
define $H_n$ to be the wreath product \begin{eqnarray*}H_n & = & H_{n-1} \wr C_{q_{n-1}} \\ & = &
(\cdots((C_{q_1} \wr C_{q_2}) \wr C_{q_3}) \wr \cdots )\wr C_{q_{n-1}} \\ & = & C_{q_1} \wr C_{q_2} \wr
\cdots \wr C_{q_{n-1}},\end{eqnarray*} that is, $H_n$ is the semidirect product by $C_{q_{n-1}}$ of the
direct product of $q_{n-1}$ copies of $H_{n-1}$, where $C_{q_{n-1}}$ acts on this direct product by
cyclic permutation of coordinates.  Note that $H_n$ is a finite group of order
\begin{equation}\label{e:orderHn} |H_n| = q_1^{q_2 q_3 \cdots q_{n-1}} q_2^{q_3 \cdots q_{n-1}} \cdots
q_{n-2}^{q_{n-1}}q_{n-1}.\end{equation} We define $H_\infty$ to be the infinite iterated (unrestricted)
wreath product \[ H_\infty:= C_{q_1} \wr C_{q_2} \wr \cdots \wr C_{q_{n-1}} \wr \cdots \] We then have
natural inclusions \[H_1 < H_2 < \cdots < H_n < \cdots < H_\infty.\]  The
following lemma will be needed for the proof of Corollary~\ref{c:infinite_generation} in
Section~\ref{ss:corollary_proof} below.

\begin{lemma}\label{l:not_fg} The group $H_\infty$ is not finitely generated.\end{lemma}

\begin{proof} By definition of
$H_\infty$, for any nontrivial $h \in H_\infty$ there is an $n \geq 1$ such that $h \in H_n$.
\end{proof}

We now define the actions of $H_n$ and $H_\infty$ on $Y_n$ and $Y_\infty$ respectively.  This uses the
label-preserving automorphisms $\alpha_n
\in A$.  Note that the action of $A$ on the nerve $L$ extends to the chamber $K$, fixing the vertex of type
$\emptyset$.  This action does not in general have a strict fundamental domain.  Inconveniently, this
action also does not satisfy Condition~\eqref{i:no_inversions} of Definition~\ref{d:action_on_scwol}
above, since for any nontrivial $\alpha \in A$, there is an edge $a$ of $K$ with $i(a)$ of type
$\emptyset$ but $\alpha(a) \neq a$.  However, to satisfy Definition~\ref{d:action_on_scwol}, it suffices
to define actions on $Y_n$ and $Y_\infty$, and then extend in the obvious way to the scwols which are
the barycentric subdivisions of these spaces, with naturally oriented edges.

For each $n \geq 1$ fix a generator $a_n$ for the cyclic group $C_{q_n}$.  Recall that $\alpha_n \in A$
has order $q_n$.  Thus for any $\alpha \in A$, there is a faithful representation $C_{q_n} \to A$, given
by $a_n \mapsto \alpha\alpha_n\alpha^{-1}$.  Recall also that $\alpha_n$ fixes the star in $L$ of the
vertex $s_{n+1}$, and that $\langle \alpha_n \rangle$ acts freely on the $\langle \alpha_n \rangle$--orbit of $s_n$.  Hence $a_n \mapsto \alpha\alpha_n\alpha^{-1}$
induces an action of $C_{q_n}$ on the chamber $K$, which fixes pointwise the mirror of type
$\alpha(s_{n+1})$, and permutes cyclically the set of mirrors of types $\alpha\alpha_n^{j_n}(s_n)$, for
$0 \leq j_n < q_n$.

We define the action of $H_n$ on $Y_n$ inductively, as follows.  The group $H_1$ is trivial.  For $n
\geq 2$, assume that the action of $H_{n-1}$ on $Y_{n-1}$ has been given.  The subgroup $C_{q_{n-1}}$ of
$H_n$ then fixes the chamber $w_{n}K=s_1 s_2 \cdots s_{n-1}K$ of $Y_n$ setwise, and acts on this chamber
via $a_{n-1} \mapsto \alpha_{n-1}$.  By the discussion above, this action fixes pointwise the mirror of
type $s_{n}$ of $w_nK$, and permutes cyclically the $q_{n-1}$ mirrors of types
$\alpha_{n-1}^{j_{n-1}}(s_{n-1})$, with $0 \leq j_{n-1} < q_{n-1}$, along which (by Lemma~\ref{l:subYn}
above), $q_{n-1}$ disjoint subcomplexes of $Y_n$, each isomorphic to $Y_{n-1}$, are attached.

By induction, a copy of $H_{n-1}$ in $H_n$ acts on each of these copies of $Y_{n-1}$ in $Y_n$.  More precisely, for $0 \leq j_{n-1} < q_{n-1}$, the $j_{n-1}$st copy of $H_{n-1}$ in $H_n$ acts on the subcomplex $Y^{j_{n-1}}_{n-1}$ of Lemma~\ref{l:subYn} above.  This action is given by conjugating the (inductively defined) action of $H_{n-1}$ on $Y_{n-1} \subset Y_n$ by the isomorphism $F^{j_{n-1}}:Y_{n-1} \to Y^{j_{n-1}}_{n-1}$ in Lemma~\ref{l:subYn}.  By definition, the action of $C_{q_{n-1}}$ cyclically permutes the subcomplexes $Y^{j_{n-1}}_{n-1}$, and so we have defined an action of $H_n$ on $Y_n$.
The action of $H_\infty$ on $Y_\infty$ is similar.

We now describe the fundamental domains for these actions.  For each $n \geq 1$ and each $1 \leq k \leq
n$, observe that $H_n$ acts transitively on the set of chambers $Y_{k,n}$.  Let $K_1 = K$, and for $n
\geq 2$ let $K_n$ be the quotient of the chamber $w_nK = s_1 s_2 \cdots s_{n-1}K$ by the action of
$C_{q_{n-1}}\leq H_n$ as defined above.  In $K_n$, the mirrors of types
$\alpha_{n-1}^{j_{n-1}}(s_{n-1})$, for $0 \leq j_{n-1} < q_{n-1}$, have been identified.  By abuse of
notation, we refer to these identified mirrors as the mirror of type $s_{n-1}$ of $K_n$.  Note also that
$C_{q_{n-1}}\leq H_n$ fixes pointwise the mirror of type $s_n$ of $w_nK$, and so we may speak of the
mirror of type $s_n$ of $K_n$.  Then a fundamental domain for the action of $H_n$ on $Y_n$ is the finite
complex \[ Z_n := \left(K_1 \cup K_2 \cup \cdots \cup K_n\right) / \sim, \] where $\sim$ means we
identify the $s_{i-1}$--mirrors of $K_{i-1}$ and $K_{i}$, for $1 \leq i < n$.  Similarly, a fundamental
domain for the action of $H_\infty$ on $Y_\infty$ is the infinite complex \[ Z_\infty := \left(K_1 \cup
K_2 \cup \cdots \cup K_n \cup \cdots \right) / \sim.\]

Finally we describe the stabilisers in $H_n$ and $H_\infty$ of the vertices of $Y_n$ and $Y_\infty$.
Let $wK$ be a chamber of $Y_n$ or $Y_\infty$.  Then there is a smallest $k \geq 1$ such that $wK \in
Y_k$.  By construction, it follows that the stabiliser in $H_n$ or $H_\infty$ of any vertex in the
chamber $wK$ is a subgroup of the finite group $H_k$.  Hence $H_n$ and $H_\infty$ act with finite
stabilisers.  Note also that for every $n \geq 1$, the action of $H_n$ fixes the vertex of type
$\emptyset$ in the chamber $w_nK$.  We may thus speak of the vertex of type $\emptyset$ in the quotient
$K_n$ defined above.  In fact, in the fundamental domains $Z_n$ and $Z_\infty$ defined above, the vertex
of type $\emptyset$ in $K_n$, for $n \geq 1$, has a lift in $Y_n$ or $Y_\infty$ with stabiliser the
finite group $H_n$.  We observe also that the actions of $H_n$ and $H_\infty$ are faithful, since the
stabiliser of the vertex of type $\emptyset$ of $K_1=K$ is the trivial group $H_1$.
Figure~\ref{f:Zinfty} shows $Z_\infty$ and the stabilisers of (lifts of) its vertices of type $\emptyset$ for the
example in Section~\ref{ss:underlying} above.

\begin{figure}[ht]
\begin{center}
\includegraphics{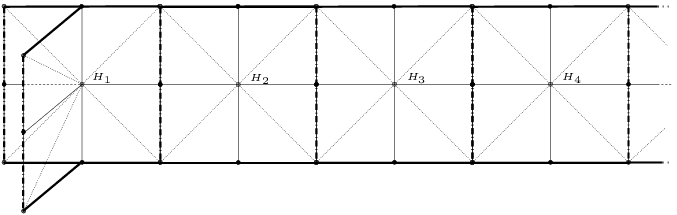}
\end{center}
\caption{Fundamental domain $Z_\infty$}
\label{f:Zinfty}
\end{figure}

\subsection{Group actions on $G(Y_n)$ and $G(Y_\infty)$}\label{ss:action}

In this section we show that the actions of $H_n$ and $H_\infty$ on $Y_n$ and $Y_\infty$, defined in Section~\ref{ss:Hn} above, extend to actions (by simple morphisms) on the complexes of groups $G(Y_n)$ and $G(Y_\infty)$.  To simplify notation, write $H$ for $H_n$ or $H_\infty$, $Y$ for $Y_n$ or $Y_\infty$, and $Z$ for $Z_n$ or $Z_\infty$.  Technically, instead of working with $G(Y)$, we work with the corresponding naturally defined complex of groups over the barycentric subdivision of $Y$, so that the action of $H$ satisfies Definition~\ref{d:action_on_scwol} above.  By abuse of notation we will however continue to write $G(Y)$.

Recall that for $\s$ a vertex of $Y$ of type $T$, the local group $G_\s$ is either $W_T$ or
$\half_s(W_T)$, and the latter occurs if and only if $\s$ is contained in an interior $s$--mirror of $Y$
with $s \in T$.  Let $wK$ be a chamber of $Y$ and let $h \in H$.  By definition of the $H$--action,
there is an $\alpha \in A$ such that for each vertex $\s$ in $wK$, with $\s$ of type $T$, the vertex
$h\cdot\s$ of $h\cdot wK$ has type $\alpha(T)$.  Moreover, if $\s$ is contained in an interior $s$--mirror then
$h\cdot\s$ is contained in an interior $\alpha(s)$--mirror.  We may thus define the local map
$\phi^h_\s:G_\s \to G_{h\cdot\s}$ by $\phi^h_\s(t) = \alpha(t)$ for each $t \in T$, and (if $G_\s =
\half_s(W_T)$), $\phi^h_\s(sts) = \alpha(s)\alpha(t)\alpha(s)$.  Then $\phi^h_\s$ is an isomorphism
either $W_T
\to W_{\alpha(T)}$, or $\half_s(W_T) \to \half_{\alpha(s)}(W_{\alpha(T)})$, as appropriate.  It is not hard to verify
that these local maps define an action of $H$ on $G(Y)$ by simple morphisms.

\subsection{Conclusion}\label{ss:conclusion}

In this section we combine the results of Sections~\ref{ss:underlying}--\ref{ss:action} above to complete the proof of the Main Theorem.

Recall that  $G(Y_1)$ is developable with universal cover $\Sigma$ (see
Section~\ref{ss:complexes_of_groups}).  By Proposition~\ref{p:covering} and
Theorem~\ref{t:coverings} above, it follows that the complexes of groups $G(Y_n)$
and $G(Y_\infty)$ are developable with universal cover $\Sigma$.  Let $H(Z_n)$ be the complex of groups induced by $H_n$ acting on
$G(Y_n)$, and $H(Z_\infty)$ that induced by $H_\infty$ acting on $G(Y_\infty)$.  By
Theorem~\ref{t:group_action} above, there are coverings of complexes of groups
$G(Y_n) \to H(Z_n)$ and $G(Y_\infty) \to H(Z_\infty)$.  Hence (by
Theorem~\ref{t:coverings} above) each $H(Z_n)$ and $H(Z_\infty)$ is developable with
universal cover $\Sigma$.

Let $\G_n$ be the fundamental group of $H(Z_n)$ and $\G$ the fundamental group of
$H(Z_\infty)$.  Since the complexes of groups $G(Y_n)$ and $G(Y_\infty)$ are
faithful, and the actions of $H_n$ and $H_\infty$ are faithful,
Theorem~\ref{t:group_action} above implies that $H(Z_n)$ and $H(Z_\infty)$ are
faithful complexes of groups.  Thus $\G_n$ and $\G$ may be identified with subgroups
of $G=\Aut(\Sigma)$.  Now $G(Y_n)$ and $G(Y_\infty)$ are complexes of finite groups,
and the $H_n$-- and $H_\infty$--actions have finite vertex stabilisers.  Hence by
construction, $H(Z_n)$ and $H(Z_\infty)$ are complexes of finite groups.  Therefore
$\G_n$ and $\G$ are discrete subgroups of $G$.  Since the fundamental domain $Z_n$
is finite, it follows that each $\G_n$ is a uniform lattice.  To show that $\G$ is a
nonuniform lattice, we use the normalisation of Haar measure $\mu$ on
$G=\Aut(\Sigma)$ defined in Section~\ref{ss:lattices} above, with the $G$--set $V$
the set of vertices of $\Sigma$ of type $\emptyset$.  Since the local groups of
$H(Z_\infty)$ at the
vertices of type $\emptyset$ in $Z_\infty$ are $H_1$, $H_2$, \ldots, we have \[ \mu(\G \bs G) =
\sum_{n=1}^\infty \frac{1}{|H_n|}. \] This series converges (see
Equation~\eqref{e:orderHn} above for the order of $H_n$, and note that each $q_n
\geq 2$).  We conclude that $\G$ is a nonuniform lattice in $G$.  Moreover, as the
covolumes of the uniform lattices $\G_n$ are the partial sums of this series, we
have $\mu(\G_n \bs G) \to \mu(\G \bs G)$, as required.  This completes the proof of
the Main Theorem.

\subsection{Proof of Corollary~\ref{c:infinite_generation}}\label{ss:corollary_proof}

The nonuniform lattice $\G$ is the fundamental group of the complex of groups
$H(Z_\infty)$ induced by the action of $H_\infty$ on $G(Y_\infty)$.  By the short exact sequence in 
Theorem~\ref{t:group_action} above, there is a surjective
homomorphism $\G \to H_\infty$.  Since $H_\infty$ is not finitely
generated (Lemma~\ref{l:not_fg} above), we conclude that $\G$ is not finitely 
generated.

\section{Examples}\label{s:examples}

In this section we describe several infinite families of examples to which the Main
Theorem applies.  By the \emph{dimension} of the Davis complex $\Sigma$ for a
Coxeter system $(W,S)$, we mean the maximum cardinality of a spherical subset of
$S$.  We note that there may be maximal spherical special subgroups $W_T$ with $|T|$
strictly less than $\dim(\Sigma)$.

\subsection{Two-dimensional examples}

If $\dim(\Sigma) = 2$ then the nerve of the Coxeter system $(W,S)$ is a graph $L$ with
vertex set $S$ and two vertices $s$ and $t$ joined by an edge if and only if $m_{st}$ is
finite.  Assume for simplicity that for some integer $m \geq 2$ all finite $m_{st} = m$.  Then
$\Sigma$ is the barycentric subdivision of a polygonal complex $X$, with all $2$--cells of $X$
regular Euclidean $2m$--gons, and the link of every vertex of $X$ the graph $L$.  Such an $X$ is
called a \emph{$(2m,L)$--complex}.   Condition~\eqref{c:halvable} of the Main Theorem can hold
only if $m$ is even, and so we also assume this.  It is then not hard to find graphs $L$ so that, for
some pair $s_1$ and $s_2$ of non-adjacent vertices of $L$, and for some nontrivial elements $\alpha_1,
\alpha_2 \in \Aut(L)$, Conditions~\eqref{c:fix},~\eqref{c:free} and~\eqref{c:orbit} of the Main Theorem also hold.  We
present three infinite families of examples.

\subsubsection{Buildings with complete bipartite links}\label{sss:right-angled}

Let $L$ be the complete bipartite graph $K_{q,q'}$, with $q,q' \geq 2$.  If $q \geq
3$ then there are (nonadjacent) vertices $s_1$ and $s_2$ of $L$, and nontrivial elements $\alpha_1$ and $\alpha_2$ of $\Aut(L)$, so that the Main Theorem applies.

If $m = 2$ then $\Sigma$ is the barycentric subdivision of the product of trees $T_q
\times T_{q'}$, where $T_q$ is the $q$--regular tree.  In particular, if $m = m' =
2$ in Example 1 of Section~\ref{ss:Davis_complexes} above, then $\Sigma$ is the
barycentric subdivision of $T_3 \times T_2$.  If $m \geq 4$, then by Theorem~12.6.1
of~\cite{D} the complex $\Sigma$ may be metrised as a piecewise hyperbolic
$\CAT(-1)$ polygonal complex.  With this metric, if $p = 2m$ and $q = q'$ then $\Sigma$ is the barycentric subdivision of Bourdon's building $I_{p,q}$
(studied in, for example,~\cite{B1} and~\cite{BP}), which is the unique $2$--complex
with all $2$--cells regular right-angled hyperbolic $p$--gons $P$, and the link of
every vertex the complete bipartite graph $K_{q,q}$.  Bourdon's building is a
right-angled hyperbolic building, of type $(W',S')$ where $W'$ is the Coxeter group
generated by the set of reflections $S'$ in the sides of $P$.

\subsubsection{Fuchsian buildings}\label{sss:fuchsian}

A Fuchsian building is a $2$--dimensional hyperbolic building.  Bourdon's building
$I_{p,q}$ is a (right-angled) Fuchsian building.  For Fuchsian buildings which are
not right-angled see,
for example,~\cite{B2} and~\cite{GP}.

To show that the Main Theorem applies to certain Fuchsian buildings which are not
right-angled, let $L$ be the finite building of rank $2$ associated to a Chevalley group $\cG$ (see~\cite{R}).  Then
$L$ is a bipartite graph, with vertex set say $S = S_1 \sqcup S_2$, and for some $k \in \{3,4,6,8\}$,
$L$ has girth $2k$ and diameter $k$.  Figure~\ref{f:proj_plane} depicts the building $L$ for the group
$\cG = GL(3,\F_2) = GL(3,2)$, for which $k = 3$.  The white vertices of this building may be identified
with the set of one-dimensional subspaces of the vector space $V=\F_2 \times \F_2 \times \F_2$, and the
black vertices with the set of two-dimensional subspaces of $V$.  Two vertices are joined by an edge if
those two subspaces are incident.

\begin{figure}[ht]
\begin{center}
\includegraphics{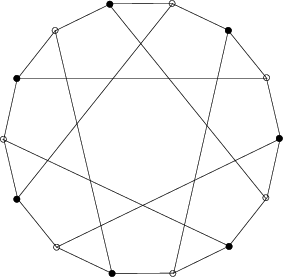}
\end{center}
\caption{The building $L$ for $\cG = GL(3,2)$}
\label{f:proj_plane}
\end{figure}

The group $\cG$ acts on $L$, preserving the type of vertices, with quotient an edge.  Suppose $s_1 \in
S_1$, and let $s_2 \in S_2$ be a vertex at distance $k$ from $s_1$.  Since $L$ is a thick building,
there is more than one such vertex $s_2$.  For $i = 1,2$, the stabiliser $P_i$ of $s_i$ in $\cG$ acts
transitively on the set of vertices of $L$ at distance $k$ from $s_i$. Now, by Theorem~6.18 of~\cite{R},
$P_i$ has a Levi decomposition \[P_i = U_i \rtimes L_i\] where $L_i$ is the subgroup of $P_i$ fixing the
vertex $s_{3-i}$.  Moreover, by Lemma 6.5 of~\cite{R}, $U_i$ fixes the star of $s_i$ in $L$.  Hence we
may find elements $\alpha_{3-i} \in U_{i}$ for which
Conditions~\eqref{c:fix} and~\eqref{c:free} of the Main Theorem hold. Condition~\eqref{c:orbit} of the Main Theorem follows
since $L$ is bipartite and the action of $\cG$ preserves the type of vertices.   For example, for $L$ as
in Figure~\ref{f:proj_plane}, if $s_1$ is the vertex $\{ (1,0,0) \}$, we may choose $s_2$ to be the
vertex $\{ (0,1,0),(0,0,1),(0,1,1)\}$, and then choose \[\alpha_1 = \left(\begin{array}{ccc} 1 & 0 & 0
\\ 1 & 1 & 0 \\ 1 & 0 & 1\end{array}\right) \quad\mbox{and}\quad\alpha_2 = \left(\begin{array}{ccc} 1 &
1 & 1 \\ 0 & 1 & 0 \\ 0 & 0 & 1\end{array}\right).\]

Suppose now that $L$ as above is the nerve of a Coxeter system $(W,S)$.  By
Theorem~12.6.1 of~\cite{D}, since $L$ has girth $\geq 6$, the corresponding Davis
complex $\Sigma$ may also be metrised as a piecewise hyperbolic $\CAT(-1)$ polygonal
complex.  With this metrisation, $\Sigma$ is then the barycentric subdivision of a
Fuchsian building, with the link of every vertex $L$ and all $2$--cells regular
hyperbolic $2m$--gons (of vertex angle $\frac{\pi}{k}$).  We call such a building a
\emph{$(2m,L)$--building}.  In general, there may be uncountably many isomorphism
classes of $(2m,L)$--buildings (see for instance~\cite{GP}).  In fact, the Davis
complex $\Sigma$ is the barycentric subdivision of the unique locally reflexive
$(2m,L)$--building with trivial holonomy (see Haglund~\cite{H2}).

\subsubsection{Platonic polygonal complexes}

A polygonal complex $X$ is \emph{Platonic} if $\Aut(X)$ acts transitively on the set of flags
(vertex, edge, face) in $X$.  Any Platonic polygonal complex is a $(k,L)$--complex,
with $k \geq 3$ and $L$ a graph such that $\Aut(L)$ acts transitively on the set of oriented
edges in $L$.  In~\cite{Sw}, \Swiatkowski\ studied CAT(0) Platonic polygonal complexes $X$,
where $L$ is a trivalent graph.  Such complexes are not in general buildings.

A graph $L$ is said to be \emph{$n$--arc regular}, for some $n \geq 1$, if $\Aut(L)$ acts simply
transitively on the set of edge paths of length $n$ in $L$.  For example, the Petersen graph in
Figure~\ref{f:petersen} above is $3$--arc regular. Any finite, connected, trivalent graph $L$,
with $\Aut(L)$ transitive on the set of oriented edges of $L$, is $n$--arc regular for some $n
\in \{1,2,3,4,5\}$ (Tutte~\cite{T}).  \Swiatkowski~\cite{Sw} showed that if $n \in \{3,4,5\}$,
then for all $k \geq 4$ there is a unique $(k,L)$--complex $X$, with $X$ Platonic.  Thus if
$k=2m$ is even, the barycentric subdivision of $X$ is the Davis complex $\Sigma$ for $(W,S)$,
where $(W,S)$ has nerve $L$ and all finite $m_{st} = m$.

Now suppose $L$ is a finite, connected, trivalent, $n$--arc regular graph with $n \in \{
3,4,5\}$.  Choose vertices $s_1$ and $s_2$ of $L$ at distance two in $L$ if $n = 3,4$, and at
distance three in $L$ if $n = 5$.  Then by Propositions 3--5 of Djokovi\'c--Miller~\cite{DM}, for
$i = 1,2$ there are involutions $\alpha_i \in \Aut(L)$ such that $\alpha_i$ fixes the star of
$s_{3-i}$ in $L$, and $\alpha_i(s_i) \neq s_i$ is not adjacent to $s_i$.  Thus if $m$
is even, the Main Theorem applies to $G=\Aut(\Sigma)$.

\subsection{Higher-dimensional examples}

We now discuss examples in dimension $> 2$ to which the Main Theorem applies.  The construction of the building $\Sigma$ below was suggested by an anonymous referee (our own examples were just for $W$ right-angled). 

We first discuss when Condition~\eqref{c:halvable} in the Main Theorem can hold.  Suppose $W_T$ is a
spherical special subgroup of $W$, with $k = |T| > 2$.  If $W_T$ is irreducible, then from the
classification of spherical Coxeter groups (see, for example,~\cite{D}), it is not hard to verify that
$W_T$ is halvable along $s \in T$ if and only if $W_T$ is of type $B_k$, with $s \in T$ the unique generator
so that $m_{st} \in \{2,4\}$ for all $t \in T - \{s\}$; in this case $\half_s(W_T)$ is of type $D_k$.  If $W_T$ is reducible, then so long as $s$ is contained in a direct factor $W_{T'}$, $T' \subsetneq T$, such that either $W_{T'} = \langle s \rangle \cong C_2$, $W_{T'}$ is an even dihedral group, or $W_{T'}$ is of type $B_j$ with $j < k$ and $s$ the particular generator described above, then $W_T$ will be halvable along $s$.

Now let $L$ be a thick spherical building of rank $k > 2$.  A reducible example is $L$ the join of $k$ sets of points, with each set having cardinality at least $3$.  An irreducible example is $L$ the building for a Chevalley group $\cG$ of rank $k$ over a finite field, such as $GL(k+1,2)$.  

Define a Coxeter group $W$ with nerve $L$ as follows.  Fix $\Delta$ a chamber of $L$.  Then $\Delta$ is a simplex on $k$ vertices.  Let $p: L \to \Delta$ be the projection onto this chamber.  Label the edges of $\Delta$ by the $m_{st}$ for a finite Coxeter group $V$ on $k$ generators, such that $V$ is a product of cyclic groups of order $2$, even dihedral groups and copies of $B_j$, $j < k$.  For example, when $V$ is right-angled all $m_{st} = 2$.  Pull the edge labels of $\Delta$ back via $p$ to obtain a labelling of the edges of $L$.  This defines a Coxeter group $W$ with nerve $L$, so that each maximal spherical special subgroup of $W$ is isomorphic to $V$.  

The Davis complex $\Sigma$ for $W$ is tiled by copies of the barycentric subdivision of the Coxeter polytope $P$ associated to $V$.  For example, when $V$ is right-angled, $P$ is a $k$--cube.  The link of each vertex of $P$ is $L$.  Applying the metric criterion of Charney--Lytchak~\cite{CL}, it follows that $\Sigma$ is the barycentric subdivision of a building.  Note that $\dim(\Sigma) = k > 2$.  

Choose vertices $s_1$ and $s_2$ in $L$ which are
\emph{opposite} (see~\cite{R}).  By the same arguments as in
Section~\ref{sss:fuchsian} above, there are (type-preserving) elements $\alpha_1,\alpha_2 \in \Aut(L)$ so that Conditions~\eqref{c:fix}--\eqref{c:orbit} of the Main Theorem hold.  A careful choice of $V$, such that $s_1$ and $s_2$ if contained in some copy of $B_j$ are both the required generators, then guarantees that Condition~\eqref{c:halvable} of the Main Theorem holds.  Hence the Main Theorem applies to many examples of buildings of dimension $> 2$.

We do not know of any \emph{hyperbolic} buildings of dimension $>2$ to
which the Main Theorem applies.  For the $3$--dimensional constructions of
Haglund--Paulin in~\cite{HP2}, certain of the $m_{st}$ must be equal to $3$, so
Condition~\eqref{c:halvable} of the Main Theorem will not hold.  

A slight modification of the above construction, for example by adding a vertex $s$ to $L$ with $m_{st} = \infty$ for all $t \in S - \{ s\}$, produces nerves which are not
buildings, hence examples of $\Sigma$ of dimension $> 2$ which are not buildings.

\end{document}